\documentclass[11pt]{amsart}

\usepackage[margin=1.2in]{geometry}

\usepackage[font=small,labelfont=bf,
justification=justified,
format=plain]{caption}

\usepackage{verbatim}

\usepackage[shortlabels]{enumitem}
\setlist{leftmargin=1.4em, itemsep=2pt, topsep=3pt, parsep=0pt}

\usepackage{amsmath,amssymb,amsthm,graphicx,enumitem,subcaption}
\usepackage{wrapfig}
\usepackage{adjustbox}
\usepackage{caption}

\usepackage[dvipsnames]{xcolor}

\definecolor{darkgreen}{rgb}{0,0.50,0}
\definecolor{darkred}{rgb}{0.55,0,0}
\definecolor{darkblue}{rgb}{0,0,0.6}
\definecolor{darkteal}{rgb}{0.0,0.25,0.5}

\usepackage[pdfborder=0,pagebackref,colorlinks,
citecolor=darkgreen,
linkcolor=darkteal,
urlcolor=darkblue]{hyperref}

\usepackage{tikz,tikz-cd}

\usetikzlibrary{
matrix,
arrows,
decorations.pathmorphing,
cd,
patterns,
calc,
backgrounds,
shapes.geometric
}

\tikzset{dummy/.style={circle,fill,draw,inner sep=0pt,minimum size=1.2mm}}
\tikzset{vertex/.style={fill,circle,minimum size=.1cm,inner sep=0pt}}

\usepackage[capitalise]{cleveref}

\crefname{convention}{Convention}{Convention}
\crefname{lemma}{Lemma}{Lemmas}
\crefname{theorem}{Theorem}{Theorems}
\crefname{definition}{Definition}{Definitions}
\crefname{proposition}{Proposition}{Propositions}
\crefname{remark}{Remark}{Remarks}
\crefname{corollary}{Corollary}{Corollaries}
\crefname{equation}{Equation}{Equations}

\theoremstyle{plain}

\newtheorem*{theorem*}{Theorem}

\newtheorem{theorem}{Theorem}[section]
\newtheorem{lemma}[theorem]{Lemma}
\newtheorem{corollary}[theorem]{Corollary}
\newtheorem{proposition}[theorem]{Proposition}

\theoremstyle{definition}

\newtheorem{notation}{Notation}
\newtheorem{construction}[theorem]{Construction}
\newtheorem{algorithm}[theorem]{Algorithm}
\newtheorem{definition}[theorem]{Definition}
\newtheorem{example}[theorem]{Example}

\newtheorem{remark}[theorem]{Remark}

\newtheorem{convention}[theorem]{Convention}

\numberwithin{equation}{section}

\newcommand{\End}{\mathrm{End}}
\newcommand{\TFS}{\mathbf{TFS}}
\newcommand{\NF}{\mathrm{NF}}

\newcommand{\id}{{\mathrm{id}}}

\newcommand{\RR}{\mathbb{R}}

\usepackage[textwidth=25mm, textsize=tiny]{todonotes}

\title[A normal form for the Thin Flat Surfaces category]
{A normal form for the Thin Flat Surfaces category}

\author[Herrera]{Jaime Herrera}

\address{
Department of Mathematics and Statistics\\
Loyola University Chicago\\
1032 W. Sheridan Road\\
Chicago, IL 60660, USA
}

\email{jherrera19@luc.edu}

\author[Rovi]{Carmen Rovi}

\address{
Department of Mathematics and Statistics\\
Loyola University Chicago\\
1032 W. Sheridan Road\\
Chicago, IL 60660, USA
}

\email{crovi@luc.edu}



\begin{document}

\begin{abstract}
The category $\TFS$ of thin flat surfaces, introduced by Khovanov, Qi, and
Rozansky, is a strict symmetric monoidal category whose morphisms are compact
oriented surfaces with corners arising as neighbourhoods of immersed graphs in
a strip. We prove two main results. First, a \emph{normal form theorem}: every
connected viewable tf-cobordism is equal, in $\TFS$, to a canonical composite
determined by its topological type. Second, a \emph{sufficiency theorem}: we establish a list of relations in the $\TFS$ category which are sufficient, meaning that any two composites of generators
representing the same tf-cobordism are related by a finite sequence of the
listed relations. Together these results give a complete generators-and-relations
presentation of $\TFS$.  
\end{abstract}

\maketitle
\tableofcontents

\maketitle

\section{Introduction}

The category $\TFS$ of thin flat surfaces was introduced by Khovanov, Qi, and
Rozansky~\cite{TFS} as the categorical framework in which the evaluation of
link invariants via two-dimensional cobordisms takes place. The objects of
$\TFS$ are non-negative integers, representing ordered disjoint unions of
closed intervals. Its morphisms are compact oriented surfaces with corners,
called thin flat surfaces or tf-cobordisms, arising as regular neighbourhoods
of immersed graphs in the strip $\RR \times [0,1]$. The category $\TFS$ is a
strict symmetric monoidal category, and its relevance to Khovanov homology and
related link invariants rests on the possibility of defining symmetric monoidal
functors out of it: an assignment of algebraic data to the generators of $\TFS$
that is consistent with the relations of the category determines such a functor
and hence an invariant of links, \cite{Seong-Khovanov}, \cite{Seong-Khovanov2}. In this setting, tf-cobordisms govern the maps
between homology groups induced by cobordisms between links, and any algebraic
object satisfying the $\TFS$ relations gives a consistent evaluation of those maps.

The $\TFS$ admits a description by generators and
relations. The generators are five basic morphisms: the unit $\iota : 0 \to
1$, the counit $\varepsilon : 1 \to 0$, the multiplication $\mu : 2 \to 1$,
the comultiplication $\Delta : 1 \to 2$, and the permutation $P : 2 \to 2$. (We also refer to permutations as twists throughout).
These generators satisfy a list of relations reflecting the topology of the
surfaces. What, to our knowledge, is not addressed in the current Literature is
whether these relations are \emph{sufficient}: that is, whether any two composites of
generators representing the same tf-cobordism are shown to be related by a finite
sequence of the listed relations. Without sufficiency, one cannot define a
functor out of $\TFS$ simply by specifying images of generators; one must
instead verify well-definedness separately on every morphism, or employ a
different mechanism altogether. The question of sufficiency is the central
question of the present paper.

\medskip

\noindent\textbf{Main results.}

The first main result is the construction of a Normal Form. 
Every connected viewable tf-cobordism $f : n \to m$ is equal, as a morphism in
$\TFS$, to a canonical composite
\begin{equation}
  \NF(f)
  \;:=\;
  \Gamma(c,\, g)
  \;\circ\;
  \bar\rho(f)
  \;\circ\;
  K_p\!\bigl(\lambda_1, \ldots, \lambda_p\bigr)
  \;\circ\;
  \bar\sigma(f).
\end{equation}
determined by the number $p$ of interval cycles, the locally cyclic partition
$\lambda = (\lambda_1, \ldots, \lambda_p)$ of the horizontal boundary
components, and the genus $g$ and number of side boundary circles $c$ of the
surface.

\begin{theorem*}[Sufficiency]
The relations of $\TFS$ are sufficient. Any two composites of the generators
$\iota, \varepsilon, \mu, \Delta, P$ representing diffeomorphic tf-cobordisms
relative to their horizontal boundary are related by a finite sequence of
applications of the listed relations.
\end{theorem*}

The four components of the normal form are: an input permutation
$\bar{\sigma}(f)$, which reorders the incoming intervals into their canonical
arrangement; the connected core $K_p(\lambda_1, \ldots, \lambda_p)$, which
assembles $p$ minimal cycle surfaces linked by $(p-1)$ connector morphisms
$b_3$; a topology coupon $\Gamma(c, g)$, which appends $c$ hole morphisms
$b_1$ and $g$ genus morphisms $b_2$; and an output permutation $\bar{\rho}(f)$,
which places the outgoing intervals in their canonical order. Each piece is
defined precisely in Section~\ref{sec:normalform}. The sufficiency theorem is
proved constructively: we give an explicit algorithm that transforms any
composite of generators representing a connected viewable tf-cobordism into
normal form using only the listed relations. Since the normal form is determined
by topological invariants of the surface, any two representations of the same
cobordism reduce to the same normal form and therefore are related by the listed
relations.

\medskip

\noindent\textbf{Consequences.}

\noindent\emph{Complete algebraic presentation.} The normal form and sufficiency
theorem together give a complete generators-and-relations presentation of $\TFS$:
two composites of generators represent the same morphism if and only if they are
related by the listed relations. This is the difference between a list of
identities that are known to hold and a complete presentation of the category.
In practice, equality of morphisms in $\TFS$ can be decided algebraically by
reducing both sides to normal form using the relations, with no appeal to
geometric arguments about surfaces or direct verification of diffeomorphism.

\noindent\emph{Functors from generators and relations.} Once the relations are
known to be sufficient, a symmetric monoidal functor $F : \TFS \to \mathcal{C}$
is determined uniquely by specifying images of $\iota, \varepsilon, \mu, \Delta,
P$ in $\mathcal{C}$, subject only to checking that those images satisfy the
listed relations in $\mathcal{C}$. This is the standard mechanism by which a
cobordism category produces algebraic invariants, and it applies whenever the
target algebra satisfies the $\TFS$ relations. In particular, this includes
non-commutative Frobenius algebras with a non-degenerate trace, a class broader
than the scalar-parametrised evaluations of~\cite{TFS} and potentially relevant
in the categorification of non-semisimple representations, where the algebra of
interest need not arise from the universal construction.

\noindent\emph{Freeness theorem.} The normal form and sufficiency theorem imply
that $\TFS$ is the symmetric monoidal category freely generated by the algebraic
structure defined by its generators and relations. Precisely, $\TFS \cong
\mathrm{Th}(\mathcal{A})$, the theory of symmetric Frobenius algebra objects
with non-commutative multiplication, carrying an additional generator $b_3$
subject to the $\TFS$ relations. This is the $\TFS$ analogue of the central theorem
of Lauda and Pfeiffer~\cite{LaudaPfeiffer2008}, who showed that the open-closed cobordism
category $\mathbf{2Cob}^{\mathrm{ext}}$ is equivalent, as a symmetric monoidal
category, to the theory of knowledgeable Frobenius algebras.

\medskip

\noindent\textbf{Relation to the Open-closed normal form of Lauda and Pfeiffer.}
The category $\TFS$ sits inside $\mathbf{2Cob}^{\mathrm{ext}}$ as a proper
full subcategory on the interval objects. As observed in~\cite[Section~5.2]{TFS},
extending $\TFS$ by allowing circles as objects and as boundary components
recovers the category studied in~\cite{LaudaPfeiffer2008}. A natural question is therefore
whether the normal form and sufficiency theorem of~\cite{LaudaPfeiffer2008} restricts to
$\TFS$ directly. The answer is, as we understand it, no, for three independent structural reasons.

First, the Lauda--Pfeiffer normal form for an interval-to-interval morphism in
$\mathbf{2Cob}^{\mathrm{ext}}$ passes through the circle objects in an
essential way. Its second stage applies a bijection $\Lambda$ that converts
the morphism into one between circle objects before converting back; the
intermediate morphism is not a morphism in $\TFS$, and the normal form, viewed
as a composite, is not an expression in $\TFS$ generators.

Second, the Lauda--Pfeiffer sufficiency proof uses generators and relations
that have no counterpart in $\TFS$. The key step (Step~IV) removes all open
comultiplications using the Cardy condition $\mu \circ \tau \circ \Delta =
\iota \circ \iota^*$, which requires both the zipper $\iota : C \to A$ and the
cozipper $\iota^* : A \to C$. Neither generator exists in $\TFS$. This matters
in particular for the connector morphism $b_3 = \mu \circ P \circ \Delta \in
\End_{\TFS}(1)$: in $\mathbf{2Cob}^{\mathrm{ext}}$ the same surface equals
$\iota \circ \iota^*$, and the Lauda--Pfeiffer proof uses the Cardy condition
to rewrite and eliminate it. In $\TFS$ no such rewrite is available; the
identities $b_3^2 = b_1 \cdot b_3$ and the commutativity of $b_3$ with all
other morphisms (cf.~\cite[Propositions~2.1--2.2]{TFS}) must be handled
directly from the $\TFS$ relations for $\mu$, $P$, and $\Delta$. This is
part of our work in this paper that cannot be extracted from the Lauda--Pfeiffer argument.

\medskip

\noindent\textbf{Comparison with the methods of Khovanov--Qi--Rozansky.}
In~\cite{TFS}, functors out of $\TFS$ and its linearisations are defined by two
mechanisms that do not require knowing the relations to be sufficient. The first
is the \emph{universal construction}: given a multiplicative evaluation $\alpha$,
a functor is built not from images of generators but from a bilinear form defined
on all tf-surfaces simultaneously, with well-definedness guaranteed by the
topological classification of connected viewable morphisms rather than by
generators and relations. The second is the construction of \emph{skein
categories}: the desired relations are imposed by definition in a quotient
category $\mathrm{STFS}_\alpha$ rather than verified from a complete
presentation, and consistency is then checked at the level of the quotient. This
is a strictly weaker condition than sufficiency of the relations for $\TFS$
itself.

Both mechanisms are well-adapted to the goals of~\cite{TFS} and rely on the
topological classification of morphisms. As a consequence, the question of
whether the listed relations are sufficient for $\TFS$ is not raised
in~\cite{TFS}. The present paper addresses this question and establishes the
affirmative answer.

\medskip

\medskip

\noindent\textbf{Organisation.}
Section~\ref{Sec:category-description} recalls the category $\TFS$, its monoidal generators, and
the three fundamental composites $b_1$, $b_2$, $b_3$, together with the complete
list of relations. Section~\ref{sec:normalform} introduces the topological
invariants of tf-cobordisms, defines the building blocks of the normal form, and
states the normal form theorem precisely. Section~\ref{Sec:Wild-to-normal}
establishes the sufficiency theorem by carrying out the explicit reduction
algorithm that transforms any connected viewable tf-cobordism into normal form
using only the listed relations.
\medskip

Throughout, we have included many figures in the hope that they will help
the reader clarify the various algorithms and ideas in this paper. As Lewis
Carroll reminds us:

\begin{quote}
``And what is the use of a book,'' thought Alice, ``without pictures or
conversation?''~\cite[Ch.1]{Alice}.
\end{quote}

\section*{Acknowledgments}  
The first-named author gratefully
acknowledges support from the LUROP Provost Research Fellowship at Loyola University Chicago.
The second-named author gratefully
acknowledges support from the College of Arts and Sciences at Loyola University Chicago for summer
research funding, as well as the AMS and Simons Foundation for support through an AMS-Simons PUI
grant.
Both authors would like to thank You Qi and Rafael Gonz\'alez D'Le\'on for helpful and motivating conversations. 
\section{The $\TFS$ category}\label{Sec:category-description}

\subsection{The Thin Flat Surface Category}

In this section we briefly review the category of thin flat surfaces ($\TFS$)
introduced in~\cite{TFS}.  Our goal is to establish the notation and conventions
that will be used throughout the remainder of the paper.  We recall the objects,
morphisms, symmetric monoidal structure, and graphical calculus in this category. 

The objects of $\TFS$ are disjoint unions of finitely many closed intervals.  Since
only the number of interval components is relevant up to isomorphism, we denote the
disjoint union of $n$ intervals by the integer $n$.  The morphisms of $\TFS$ are
thin-flat cobordisms, defined as follows.

\begin{definition}
\label{def:tfcob}
A \textbf{thin-flat cobordism} from $n$ to $m$ is an oriented compact surface $S$
with corners whose boundary decomposes into $\partial S = \partial_v S \cup \partial_h S$,
where $\partial_v S$ consists of $n$ incoming boundary intervals and $m$ outgoing
boundary intervals, and every boundary component is subdivided into alternating
horizontal and vertical boundary arcs.  Two such tf-cobordisms are considered
equivalent if they are diffeomorphic relative to the vertical boundary.
\end{definition}

\begin{definition}
\label{def:TFS}
The category $\TFS$ is a symmetric monoidal category with
\begin{enumerate}[(1)]
  \item Objects: disjoint unions of closed intervals, and
  \item Morphisms: thin flat surfaces between them.
\end{enumerate}
The symmetric monoidal structure comes from the permutation morphism $P$, where
$P \circ P = \id \otimes \id$.
\end{definition}

\begin{remark}
Every thin-flat cobordism admits a planar representation as an immersed surface in
$\RR \times [0, 1]$.  Like in~\cite{TFS}, in this paper we identify a tf-cobordism
with its planar representation.  Any self-overlaps of the immersion can be regarded
as purely virtual.  Equivalently, self-overlaps may be drawn by arbitrarily drawing
them as over- or under-crossings without changing the represented morphism.
\end{remark}

\begin{convention}
\label{conv:ltr}
Unlike the conventions of~\cite{TFS}, throughout this paper we draw morphisms from
left to right, with incoming intervals on the left and outgoing intervals on the right.
All figures and interval labels will follow this convention unless otherwise stated.

Each planar representation of a tf-surface inherits an orientation from its immersion
in $\RR \times [0, 1]$.  We label the incoming intervals from top to bottom by
$1, \ldots, n$ and the outgoing intervals from bottom to top by $1', \ldots, m'$.
\end{convention}

By drawing our diagrams from left to right, cobordisms have incoming interval
boundaries on the left and outgoing interval boundaries on the right.  

\begin{definition}
\label{def:intcycle}
The \textbf{interval cycle} of a boundary component is the cyclic ordering of the
vertical intervals encountered while traversing that boundary component according to
its induced orientation.  For the planar diagrams of this paper, the induced
orientation on each boundary component corresponds to a counterclockwise traversal.
\end{definition}

\begin{example}
\label{ex:cycle}
Consider a connected tf-cobordism whose unique boundary component has interval cycle
\[
  (1\ 2\ 1'\ 2').
\]
This indicates that, following the induced orientation, one encounters the incoming
intervals $1$ and $2$, followed by the outgoing intervals $1'$ and $2'$, before
returning to the starting point.  Since the ordering is cyclic, any cyclic rotation,
such as $(2\ 1'\ 2'\ 1)$, represents the same interval cycle.
\end{example}

\subsection{Generators and relations of the $\TFS$ category}
\label{sec:genrel}

The $\TFS$ category can be completely described in terms of generators and relations.
Figure~\ref{Fig:generators} illustrates a complete set of generators for $\TFS$.

\begin{center}
  \includegraphics[width=0.68\textwidth]{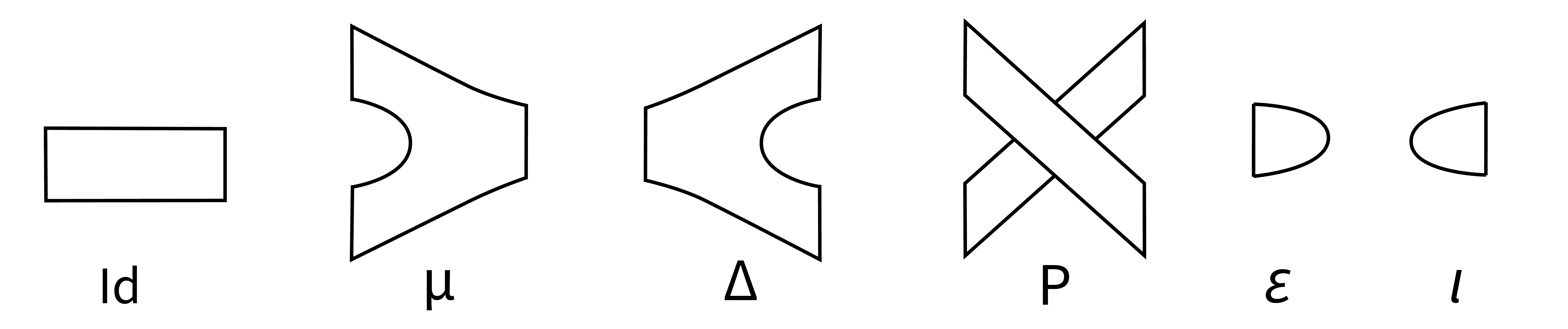}
  \captionof{figure}{Generators of the $\TFS$ category.  From left to right: Identity,
  multiplication $\mu$, comultiplication $\Delta$, permutation $P$, counit $\epsilon$,
  unit $\iota$.\label{Fig:generators}}
\end{center}

In the rest of this section, we discuss important composites and relations in the category.
\medskip

\begin{proposition}\label{Prop:relations1-10}
    The relations in Figure \ref{Fig:relations} are sufficient relations in the $\TFS$ category. 
    \begin{center}
  \includegraphics[width=0.99\textwidth]{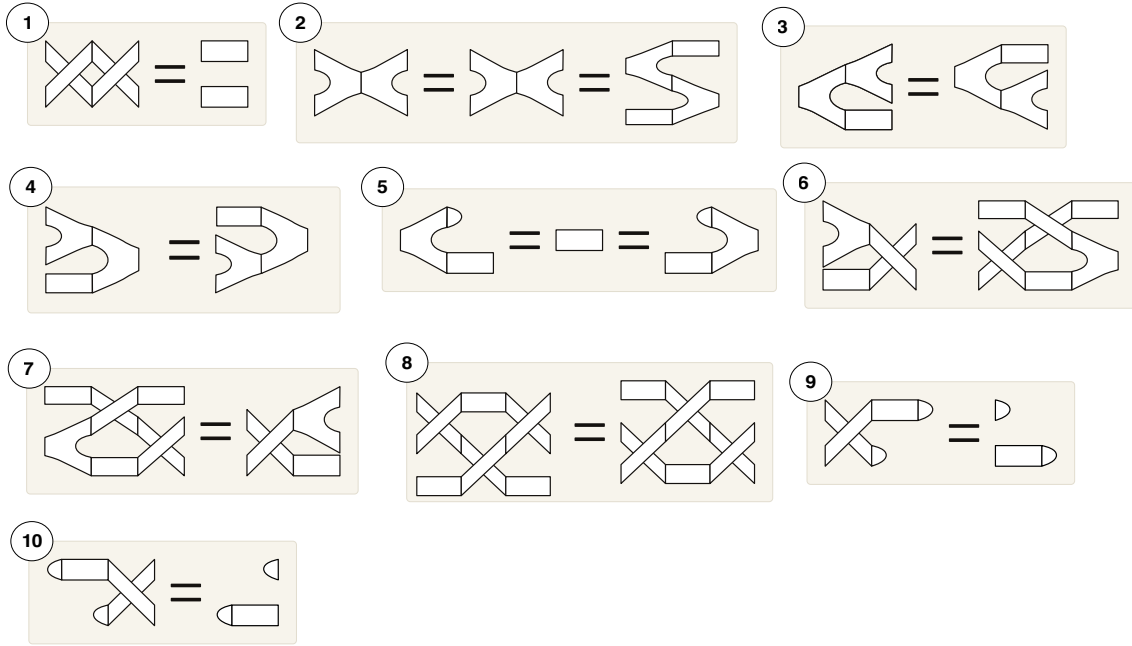}
  \captionof{figure}{Sufficient relations in the $\TFS$ category.\label{Fig:relations}}

\end{center}

\end{proposition}

In Section~\ref{Sec:Wild-to-normal} we show that these relations are indeed sufficient.
Our strategy is to show that the morphisms of the $\TFS$ category admit a normal form
determined by topological invariants, and subsequently prove that the defining relations
are sufficient to reduce any surface to this normal form.

In the remainder of this section, we will define some additional relations and establish terminology. While the additional relations can be deduced from the ones in Proposition \ref{Prop:relations1-10}, the additional relations are useful to simplify the discussion in Section \ref{Sec:Wild-to-normal}.

Relations 11 and 12 in Proposition \ref{saddle relations} pertain to saddles; see Figure \ref{Fig:saddle-thin-and-real} for the thin flat surface representation of a saddle.

\begin{center}
  \includegraphics[width=0.65\textwidth]{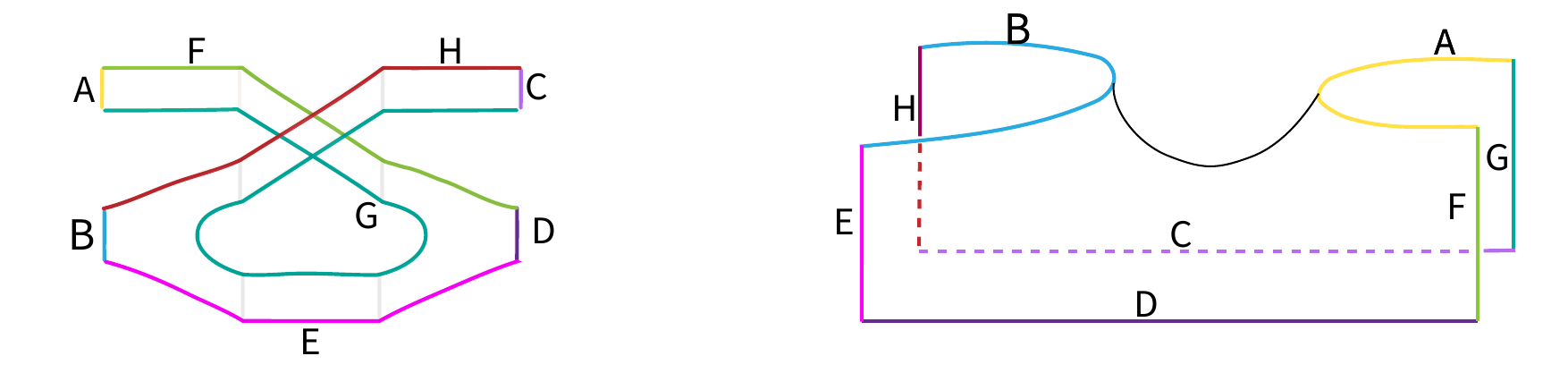}
  \captionof{figure}{Left: Saddle in terms of generators of the $\TFS$ category.
  Right: a saddle on a surface.  %
  \label{Fig:saddle-thin-and-real}}
\end{center}

Furthermore, we define the two fundamental saddle composites $S_1$ and $S_2$ (Figure~\ref{fig:S1/S2}) in this category.
\vspace{-10pt}
\begin{center}
    \includegraphics[width=0.4\textwidth]{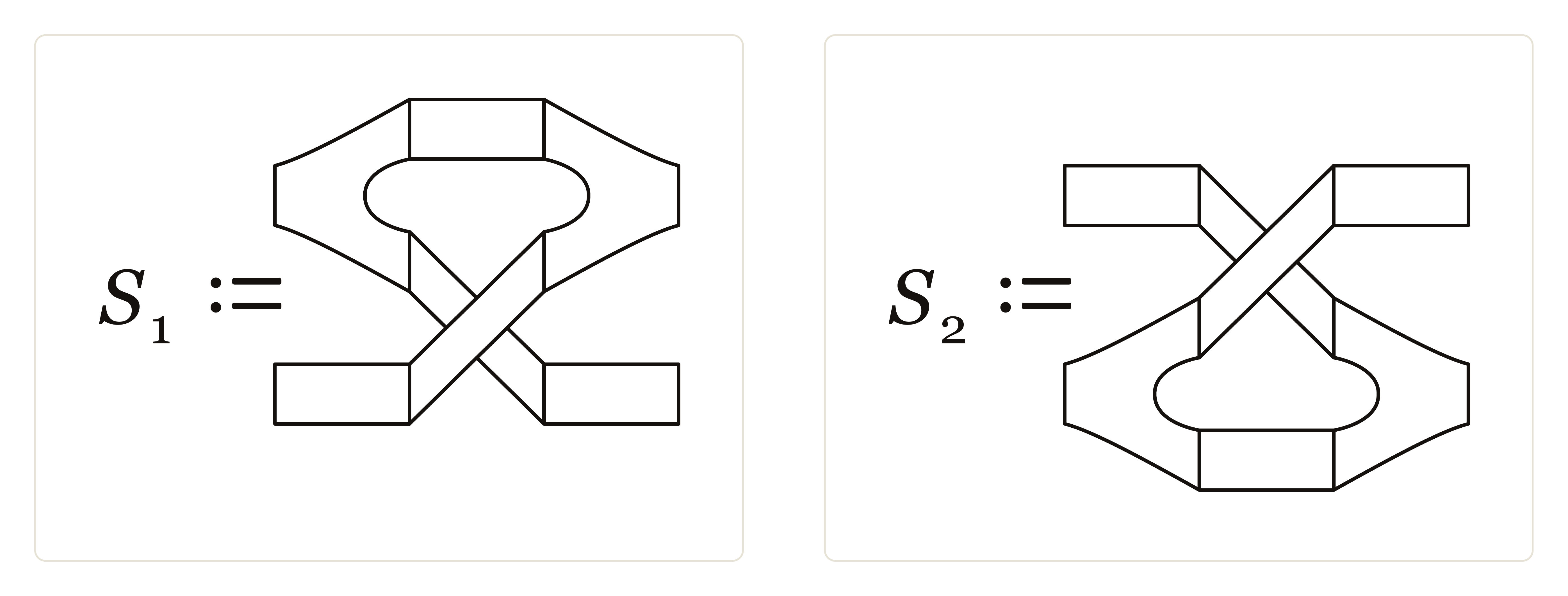}
    \captionof{figure}{Left: The $S_1$ saddle with the interval cycle ($1~1'~2~2'$). Right: The $S_2$ saddle with the interval cycle ($1~2'~2~1'$).}
    \label{fig:S1/S2}
\end{center}

\begin{remark} \label{saddle remark}
    Notice that the $\mu$'s and $\Delta$'s in $S_1$ and $S_2$ cannot be pulled past each other because they are connected via the same leg (the top legs in $S_1$ and the bottom legs in $S_2$) while non-connected legs are permuted past each other. This fact imposes restrictions on the representations of $S_1$ and $S_2$, which we discuss further in Corollary \ref{cor: saddle representations}.
\end{remark}

\begin{proposition}
\label{saddle relations}
    Relations 11 and 12 shown in Figure \ref{Fig:Relations11-12}  hold in the $\TFS$ category involving the saddles $S_1$ and $S_2$ and the permutation P:  
\begin{center}
\includegraphics[width=0.8\textwidth]{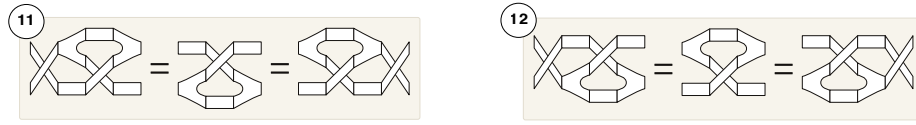}
    \captionof{figure}{Relations 11 and 12:  $S_1 \circ P = S_2 = P\circ S_1$, and
        $S_2 \circ P = S_1 = P\circ S_2$. }\label{Fig:Relations11-12}
\end{center}
\end{proposition}

\begin{proof}
    See Figure \ref{Fig:saddle-rep-sequence}.
\end{proof}

\begin{corollary}
\label{cor: saddle representations}
    After reducing to the minimal number of pants (one $\mu$ and one $\Delta$) and quotienting out the equivalence in Proposition \ref{saddle relations}, only four representations each of $S_1$ and $S_2$ exist. These are shown in Figure \ref{Fig:S1/S2 representations}:
\begin{center}
    \includegraphics[width=0.7\textwidth]{Normal/Eight-saddles.pdf}
    \captionof{figure}{}
    \label{Fig:S1/S2 representations}
\end{center}
\end{corollary}

\begin{proof}
    First, notice that the $\mu$ and $\Delta$ of a saddle can either be on the top or the bottom of its respective side, resulting in a total of four possible $\mu$, $\Delta$ configurations. The four representations of each saddle exhaust all four of these possibilities. To go from $S_1$ to $S_{1a}$, apply the equivalence in Proposition \ref{saddle relations} and place the twist on the right side. Then, using relation 8, pull the top leg of the $\Delta$ down such that it goes under the $\mu$. To go from $S_1$ to $S_{1b}$, we perform the analogous procedure on the left side. To go from $S_{1c}$ to $S_1$, using relation 1, pull the bottom legs of the $\mu$ and $\Delta$ downward. This will give $S_1$ with one pre-composed and one post-composed permutation. By Proposition \ref{saddle relations}, we reduce to $S_1$. An analogous argument applies to prove $S_2$, $S_{2a}$, $S_{2b}$, and $S_{2c}$ are equivalent.
\end{proof}

Like~\cite{TFS}, we define three fundamental composites $b_1$, $b_2$, and $b_3$
diagrammatically in Figures~\ref{Fig:b1-hole}, \ref{Fig:b2-genus},
and~\ref{Fig:b3-connector}.

The first, $b_1$, is the ``hole'' tf-cobordism which creates an additional boundary
component without any intervals.  This can be viewed as an empty cycle.

The composite $b_2$ is a genus one morphism $1 \to 1$ without additional boundary components.  

Lastly, we refer to $b_3$ as the ``connector'' morphism, which, as we will see later,
allows connected surfaces to contain multiple interval cycles.

\begin{center}
  \includegraphics[width=0.13\textwidth]{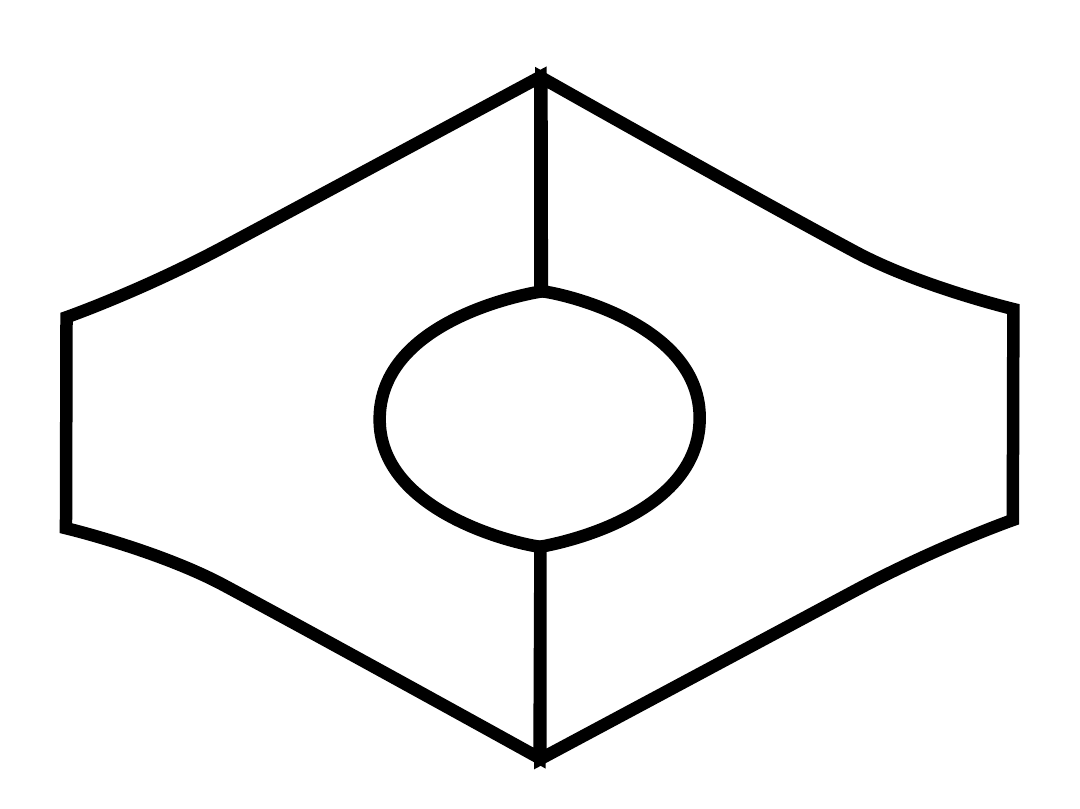}
  \captionof{figure}{Hole ($b_1$) in terms of generators of the $\TFS$ category.\label{Fig:b1-hole}}
\end{center}

\begin{center}
 \includegraphics[width=0.4\textwidth]{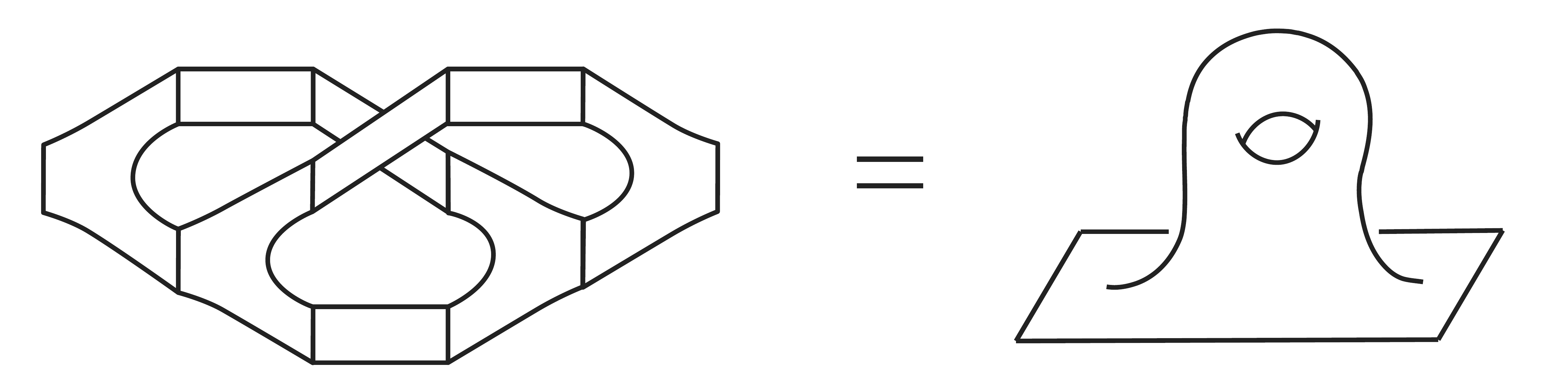}
  \captionof{figure}{Genus ($b_2$) in terms of generators of the $\TFS$ category and as
  a surface.\label{Fig:b2-genus}}
\end{center}

\begin{center}
  \includegraphics[width=0.47\textwidth]{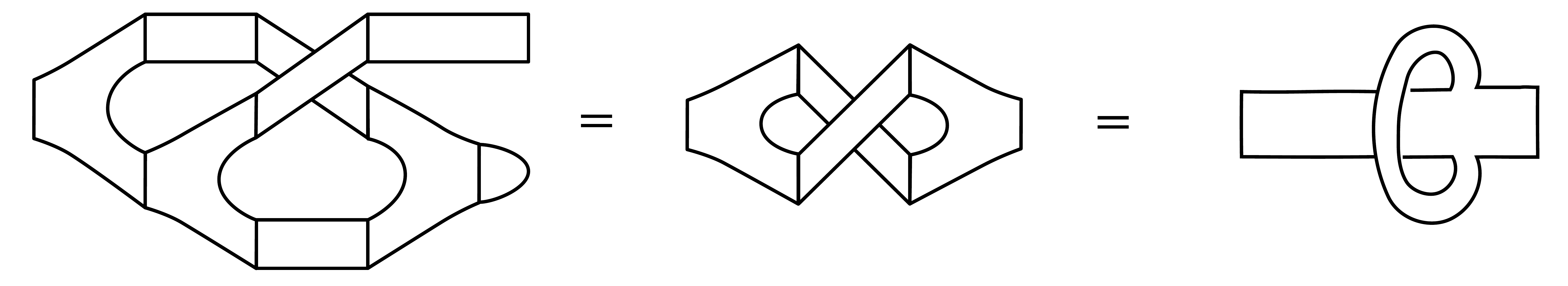}
  \captionof{figure}{The $b_3$ composite (connector) in two equivalent
  representations.\label{Fig:b3-connector}}
\end{center}

\begin{lemma}
\label{$b_2$ relation}
    The two representations of  $b_2$ in Figure \ref{Fig:two-genus} are equivalent.
\begin{center}

\includegraphics[width=0.4\textwidth]{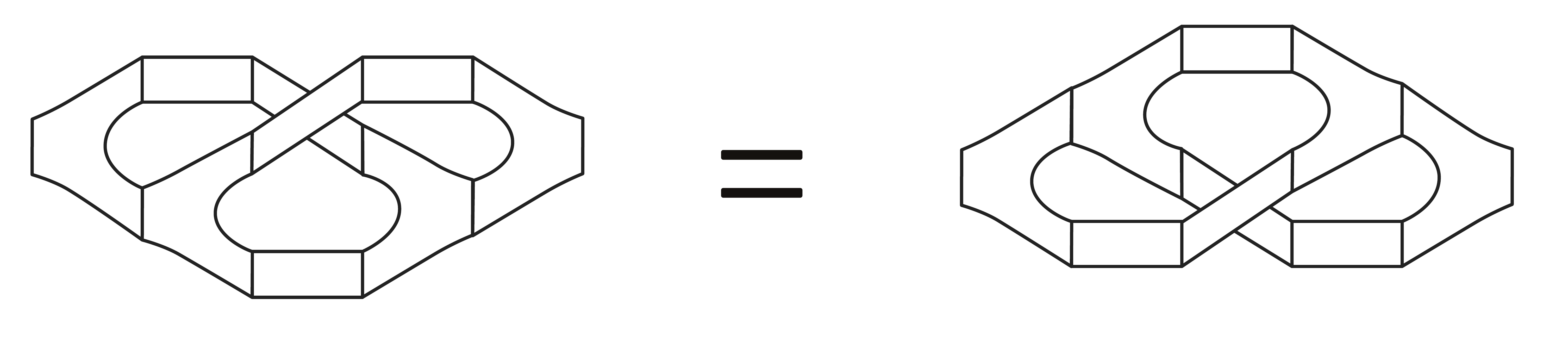}
    \captionof{figure}{}\label{Fig:two-genus}
\end{center}
\end{lemma}

\begin{proof}
    The proof that these two representations of $b_2$ are equivalent is given as a series of moves using the relations from Proposition \ref{Prop:relations1-10}.  However, the sequence is rather long, so we include it in the Appendix.
\end{proof}

\begin{corollary}
\label{cor: b2 representations}
    After reducing to the minimal number of pants (two $\mu$'s and
$\Delta$'s) and reducing saddles to either $S_1$ or $S_2$, 
$b_2$ has four representations as shown in Figure \ref{Fig:four-genus}: 

\begin{center}
    \includegraphics[width=0.4\textwidth]{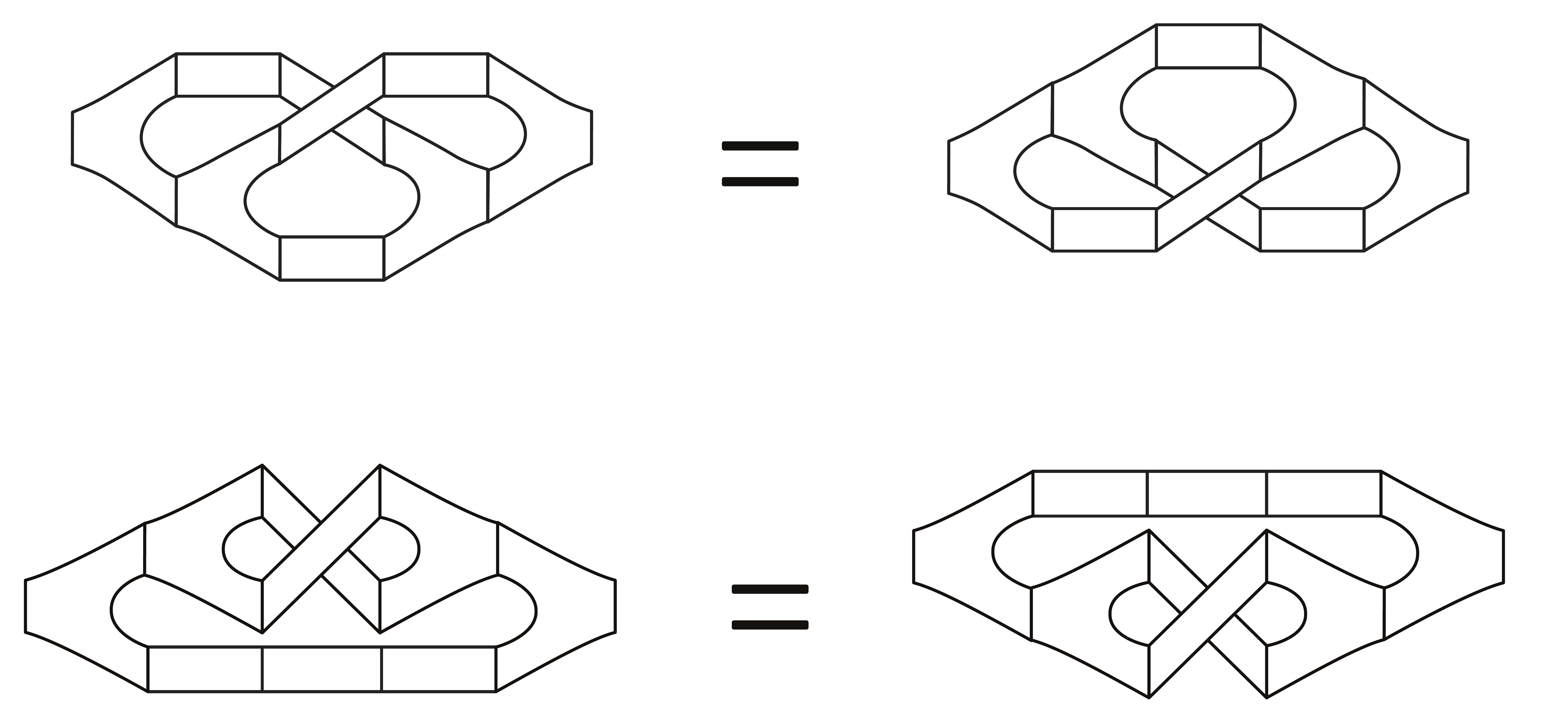}
    \captionof{figure}{Equivalent representations of the genus}\label{Fig:four-genus}
\end{center}
.
\end{corollary}

\begin{proof}
    Since the leftmost $\mu$ and the rightmost $\Delta$ of a $b_2$ cannot be moved, only the interior component between this $\mu$, $\Delta$ pair allows multiple representations. By Lemma \ref{$b_2$ relation} this interior component can be either a saddle composite, which can be reduced $S_1$ or $S_2$ by Corollary \ref{cor: saddle representations}. Any other representations of this interior component stem from applying relations 3 and 4 on the stacked $\mu$'s and $\Delta$'s, resulting in the last two composites featuring a $b_3$ either on the top or on the bottom. 
\end{proof}

\begin{lemma} \label{lem: b1, b2, b3 commutativity}
    The composites $b_1$, $b_2$, and $b_3$ pairwise commute by the relations in Figure \ref{Fig:relations}. 
\end{lemma}

\begin{proof}

Figure \ref{Fig:hole-past-b3} provides a sequence of moves using $\TFS$ relations showing that $b_3 \circ b_1 = b_1 \circ b_3$.
    \begin{center}
    \includegraphics[width=0.75\textwidth]{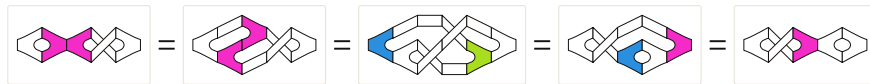}
    \captionof{figure}{The following moves show $b_3 \circ b_1 = b_1 \circ b_3$.}\label{Fig:hole-past-b3}
\end{center}

Figure \ref{Fig:hole-past-genus} provides a sequence of moves using $\TFS$ relations showing that $b_2 \circ b_1 = b_1 \circ b_2$.
\begin{center}
    \includegraphics[width=0.85\textwidth]{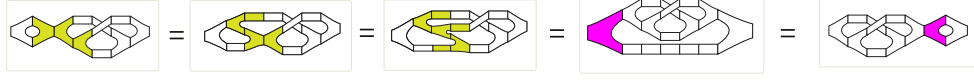}
    \captionof{figure}{The following moves show $b_2 \circ b_1 = b_1 \circ b_2$.}\label{Fig:hole-past-genus}
\end{center}
Similarly, 
Figure \ref{Fig:genus-past-b3} provides a sequence of moves using $\TFS$ relations showing that $b_3 \circ b_2 = b_2 \circ b_3$.
\begin{center}
    \includegraphics[width=0.85\textwidth]{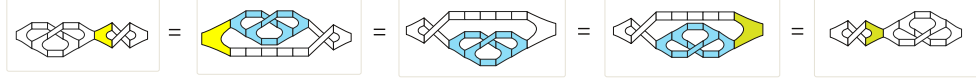}
    \captionof{figure}{The following moves show $b_3 \circ b_2 = b_2 \circ b_3$.}\label{Fig:genus-past-b3}
\end{center}

\end{proof}

\begin{lemma} \label{lem: b3 connector relations}
    The three relations involving  $b_3$ in Figure \ref{Fig:b3-equivalences} hold in the $\TFS$ category.
 \begin{center}
    \includegraphics[width=0.99\textwidth]{Normal/b3-equivalences.pdf}
    \captionof{figure}{}\label{Fig:b3-equivalences}
 \end{center}

\end{lemma}

\begin{proof} We provide a sequence of move using relations in the category showing each of these three equivalences.
  \begin{enumerate}
      \item The first case requires the following moves shown in Figure \ref{Fig:proof-b3-equivalences}. 
\begin{center}
\includegraphics[width=0.85\textwidth]{Normal/Proof-b3-equivalences.pdf}
\captionof{figure}{}\label{Fig:proof-b3-equivalences}
 \end{center}
    \item The second equivalence is an analog of (1).
    \item The third equivalence can be shown using the sequence of moves in Figure \ref{Fig:proof-b3-equivalences-case3}. 
\begin{center}
\includegraphics[width=0.85\textwidth]{Normal/Proof-b3-equivalences-part3.pdf}
    \captionof{figure}{}\label{Fig:proof-b3-equivalences-case3}
 \end{center} 
 
 \end{enumerate}
      
\end{proof}

\subsection{Thin Flat Surface Classification and Invariants}
\label{Sec:invariants}

A tf-cobordism $C$ decomposes into $C = \bigsqcup_{i=1}^k C_i$ where
$C_1, \ldots, C_k$ are the connected components of $C$.  As explained in~\cite{TFS},
for each connected component of $C$, there are three invariants that uniquely determine
the component up to diffeomorphism relative to the boundary:
\begin{enumerate}[(1)]
  \item the interval cycle(s);
  \item the number of boundary components;
  \item the genus of the surface.
\end{enumerate}

A connected tf-cobordism is \textbf{viewable} if every boundary component contains at
least one interval.  A boundary component containing no intervals is called a
\textbf{hole}.  In what follows we work exclusively with viewable tf-cobordisms.

The interval cycles, together with the genus and the number of boundary components,
form the complete set of invariants for connected viewable thin-flat cobordisms.  We
describe these invariants more explicitly.

Let $f \colon n \to m$ be a connected viewable tf-cobordism.  Traversing each boundary
component of $f$ according to its induced orientation and recording the incoming and
outgoing intervals encountered, we obtain a collection of \emph{interval cycles}
$\lambda_1, \ldots, \lambda_p$.  Each cycle $\lambda_k$ is a cyclic ordering of a
subset of the boundary labels $\{1, \ldots, n, 1', \ldots, m'\}$, where unprimed
labels refer to incoming intervals and primed labels to outgoing ones.  We write $a_k$
for the number of incoming intervals in $\lambda_k$ and $a'_k$ for the number of
outgoing intervals, so that
\[
  \sum_{k=1}^{p} a_k = n
  \qquad \text{and} \qquad
  \sum_{k=1}^{p} a'_k = m.
\]
In addition to the cycle data, $f$ has two further invariants: its \emph{genus}
$g \geq 0$ and its number of \emph{holes} $c \geq 0$.

\begin{definition}
\label{def:minimal}
A thin-flat cobordism is \textbf{minimal} if it is genus $0$ and all boundary
components contain intervals, that is, $c = 0$.
\end{definition}

\subsection{Pairs of pants required for a minimal surface}
\label{Sec:pants-minimal-surface}

The number of multiplications $\#\mu$, and comultiplications $\#\Delta$, required to express a minimal
surface is determined by the block structure of its interval cycles.

\begin{definition}
\label{def:block}
A collection of $q$ intervals $d_1 \ldots d_q$ is a \textbf{block} if they are all
either outgoing or incoming and consecutively ordered in an interval cycle.
\end{definition}

\begin{remark}
Not all intervals in a block define a chain of adjacent interval pairs.  A block in a
cycle may be ``separated'' by other intervals depending on how the cycle is written;
for example, the block $a_1 a_2 a_3$ may appear as $(a_2~a_3 \ldots a_1)$.
\end{remark}

The required $\#\mu$ for an arbitrary interval cycle $(a_1 \ldots)$ corresponds to
adding the number of incoming adjacent interval pairs with the number of blocks composed
of outgoing interval(s) between the first and last block.  In particular, for a surface
with $n$ in-boundary intervals, this gives $\#\mu \geq n - 1$.  Similarly, the required
$\#\Delta$ for an arbitrary interval cycle corresponds to adding the number of outgoing
adjacent interval pairs with the number of blocks composed of incoming interval(s)
between the first and last block.  Thus, for a surface with $m$ out-boundary intervals,
$\#\Delta \geq m - 1$.

\section{A normal form}
\label{sec:normalform}

In~\cite{TFS} it is argued that a connected, viewable thin flat cobordism is comprised
of a ``coupon'' of powers of $b_1$ (holes) and $b_2$ (genus) and a minimal connected
component.  Furthermore, the copies of $b_1$ and $b_2$ can be placed anywhere on the
surface, producing diffeomorphic results.  In our normal form, we place the copies of
$b_1$ and $b_2$ in a coupon $b_1^i b_2^j$, with $i, j \in \mathbb{Z}^+$, on the top
leg at the right of the figure.

The minimal part of a viewable connected component describes a connected surface with
a number of circles marked with in- and out-going intervals, as in
Figure~\ref{Fig:circles-in-out}.  Each one of these circles with marked intervals is
described as a cycle.

\begin{center}
\includegraphics[width=0.25\textwidth]{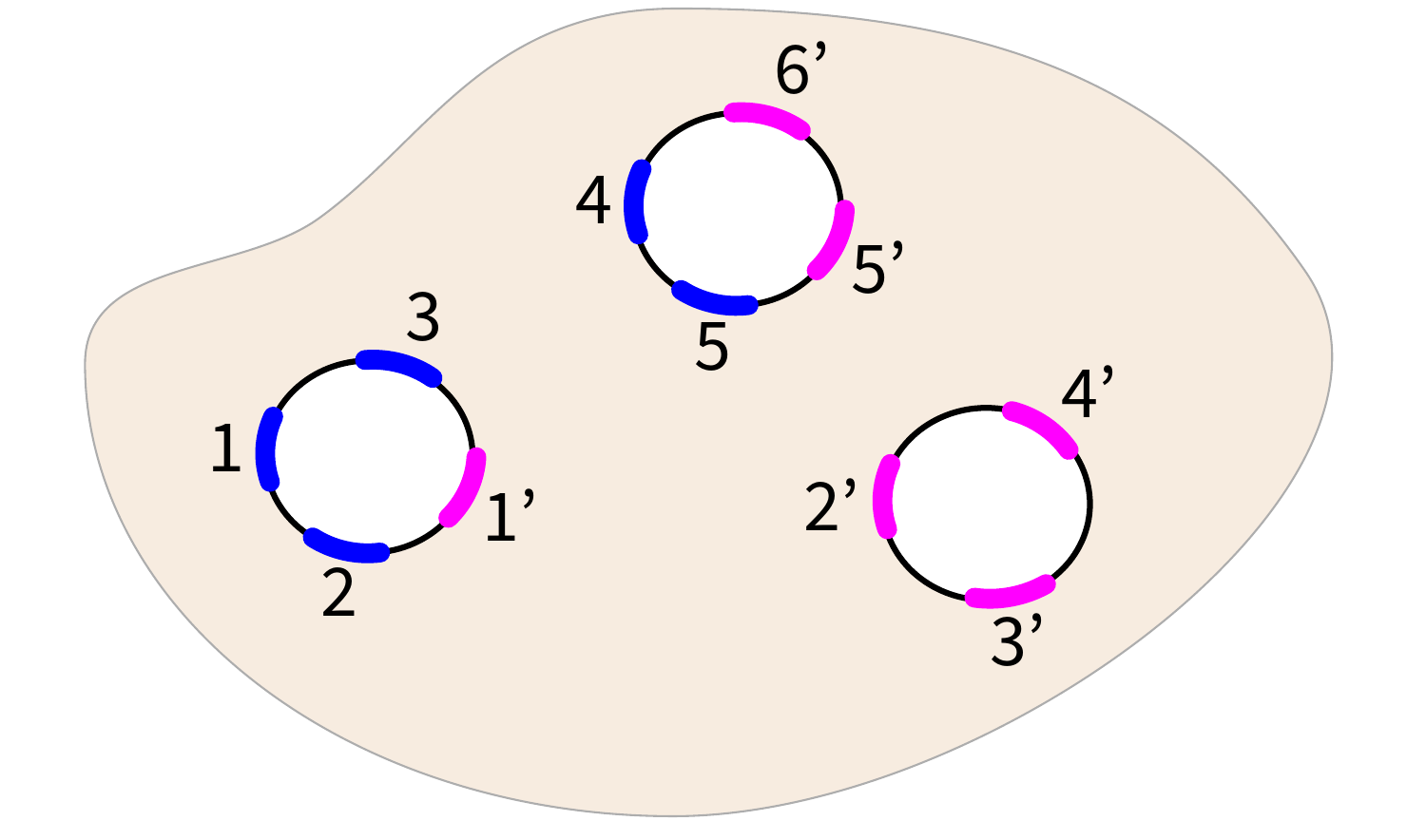}
    \captionof{figure}{Part of a minimal surface showing three cycles: $(1~2~1'~3)$, $(4~5~5'~6')$, and $(2'~3'~4')$.}\label{Fig:circles-in-out}
 \end{center}

The construction of our normal form follows a strategy in four stages:
\begin{enumerate}[(1)]
  \item Algorithm~\ref{alg:cycle-to-gen} expresses each cycle in terms of generators
        of the $\TFS$ category.
  \item We then create a connected surface out of several cycles without modifying the
        individual cycles; see Section~\ref{Sec:use-of-b3}.
  \item The order of the cycles is not a crucial part of the structure of our normal
        form.  However, we briefly discuss a choice of ordering for the cycles; see Section~\ref{sec:ordering} 
  \item The last step is adding the coupon of $b_1$'s and $b_2$'s.
\end{enumerate}

\subsection{The building blocks}
\label{sec:blocks}

We define four families of building blocks from which the normal form is assembled.
The first two are permutations of the source and target objects; the third is the
combinatorial heart of the construction; the fourth encodes the topology.

\begin{construction}[Input permutation]
\label{con:input-perm}
After fixing an ordering of the cycles, the \emph{input permutation}
$\bar\sigma(f) \colon n \to n$ is the morphism in $\TFS$ that rearranges the $n$
incoming intervals into \emph{block order}: for each $k = 1, \ldots, p$, the $a_k$
incoming intervals belonging to $\lambda_k$ are moved to the consecutive positions
\[
  a_1 + \cdots + a_{k-1} + 1, \quad \ldots, \quad a_1 + \cdots + a_k,
\]
and within this block they are ordered by the cyclic order of $\lambda_k$ starting
from the smallest label.  The morphism $\bar\sigma(f)$ is expressed as a composition
of permutation generators $P$ in $\TFS$.
\end{construction}

\begin{construction}[Output permutation]
\label{con:output-perm}
The \emph{output permutation} $\bar\rho(f) \colon m \to m$ is the morphism in $\TFS$
that reorders the $m$ outgoing intervals from the block order produced by the connected
core (defined below) to the labeling $1', \ldots, m'$ of the outgoing intervals as they
appear in $f$.  Like $\bar\sigma(f)$, it is a composition of permutation generators $P$.
\end{construction}

\begin{construction}[Minimal cycle surfaces]
\label{con:min-cycle}
For each $k = 1, \ldots, p$, the \emph{minimal cycle surface} $A(\lambda_k)$ is the
tf-cobordism
\[
  A(\lambda_k) \colon a_k \longrightarrow a'_k
\]
produced by Algorithm~\ref{alg:cycle-to-gen} applied to the cycle $\lambda_k$.  It is
the unique minimal tf-cobordism (up to the relations of $\TFS$) whose single boundary
component realizes the cycle $\lambda_k$; in particular, it is genus~$0$, has no holes,
and has exactly one boundary component.  By Section~\ref{Sec:pants-minimal-surface}, the
surface $A(\lambda_k)$ contains precisely $a_k - 1$ multiplications and $a'_k - 1$
comultiplications.
\end{construction}

\begin{construction}[Connected core]
\label{con:connected-core}
The \emph{connected core}
\[
  K_p(\lambda_1, \ldots, \lambda_p) \colon n \longrightarrow m
\]
is the connected tf-cobordism obtained from the minimal cycle surfaces
$A(\lambda_1), \ldots, A(\lambda_p)$ by applying $(p-1)$ copies of the connector
morphism $b_3$, following the procedure of Section~\ref{Sec:use-of-b3}.  
Specifically, the $(p-1)$ copies of $b_3$ connect $A(\lambda_2), \ldots, A(\lambda_p)$
directly to $A(\lambda_1)$, so that the $p$ minimal cycle surfaces are joined in a
star pattern centered at $A(\lambda_1)$; that any other arrangement of $(p-1)$
copies of $b_3$ yields the same morphism is Corollary~\ref{cor:tree-order}.

The
connection does not alter the interval structure of any individual cycle: each cycle
$\lambda_k$ retains its own boundary component in the resulting surface.  The outgoing
intervals of $K_p$ are arranged in blocks, with the $a'_k$ outgoing intervals of
$A(\lambda_k)$ appearing consecutively in positions $a'_1 + \cdots + a'_{k-1} + 1$
through $a'_1 + \cdots + a'_k$.
\end{construction}

\begin{remark}
\label{rem:connection}
The connection procedure depends on whether the cycles being connected have $\begin{cases} \textnormal{incoming} \\ \textnormal{outgoing} \end{cases}$
intervals.  When a cycle $\lambda_k$ has $\begin{cases} a_k = 0 \\ a'_k = 0 \end{cases}$, the $\begin{cases} \textnormal{unit initiating} \\ \textnormal{counit terminating} \end{cases}$  
$A(\lambda_k)$ is removed and $b_3$ is attached to the resulting free boundary before
connecting to the adjacent cycle; see Example~\ref{ex:connecting} and
Figure~\ref{Fig:cycle-b3}.  The four cases are illustrated in
Figure~\ref{Fig:four-cases-connect-b3}.
\end{remark}

\begin{construction}[Topology coupon]
\label{con:topology-coupon}
The \emph{topology coupon}
\[
  \Gamma(c, g) \colon m \longrightarrow m
\]
is the morphism that attaches $c$ holes and $g$ handles to the first outgoing interval
of the connected core.  It is defined as
\[
  \Gamma(c, g) \;:=\; (b_1^{c} \circ b_2^{g}) \otimes \id_{m-1},
\]
where $b_1, b_2 \in \End_{\TFS}(1)$ are the hole and genus endomorphisms
of~\cite{TFS}, and $\id_{m-1}$ is the identity on the remaining $m - 1$ outgoing
intervals.  When $m = 1$, this reduces to $b_1^{c} \circ b_2^{g} \colon 1 \to 1$.
Since $b_1$ and $b_2$ commute (Proposition~2.1 of~\cite{TFS}), the order of
composition within $b_1^{c} \circ b_2^{g}$ does not affect the result.
\end{construction}

\subsection{The normal form definition}
\label{sec:normalformdef}

\begin{definition}[Normal form of a connected viewable tf-cobordism]
\label{def:nf}
Let $f \colon n \to m$ be a connected viewable tf-cobordism with interval cycles
$\lambda_1, \ldots, \lambda_p$, where $\lambda_k$ involves $a_k$ incoming and $a'_k$
outgoing intervals, genus $g \geq 0$, and $c \geq 0$ holes.  Fix an ordering of the
cycles as in Section~\ref{sec:ordering}.  The \emph{normal form} of $f$ is the
tf-cobordism
\begin{equation}
\label{eq:nf}
  \NF(f)
  \;:=\;
  \Gamma(c,\, g)
  \;\circ\;
  \bar\rho(f)
  \;\circ\;
  K_p\!\bigl(\lambda_1, \ldots, \lambda_p\bigr)
  \;\circ\;
  \bar\sigma(f).
\end{equation}
For a general (not necessarily connected) viewable tf-cobordism $f = f_1 \otimes
\cdots \otimes f_c$ decomposed into connected components, the normal form is the
monoidal product
\[
  \NF(f) \;:=\; \NF(f_1) \otimes \cdots \otimes \NF(f_c).
\]
\end{definition}

Reading equation~\eqref{eq:nf} from right to left (that is, following the composition
from source to target):
\begin{enumerate}[\upshape (1)]
  \item $\bar\sigma(f)$ rearranges the $n$ incoming intervals into block order,
        grouping each cycle's incoming intervals consecutively.
  \item $K_p(\lambda_1, \ldots, \lambda_p)$ builds the connected minimal surface:
        each cycle is expressed as a minimal surface $A(\lambda_k)$ via
        Algorithm~\ref{alg:cycle-to-gen}, and the $p$ surfaces are joined using
        $(p-1)$ copies of $b_3$.
  \item $\bar\rho(f)$ reorders the $m$ outgoing intervals from block order to the
        labeling of $f$.
  \item $\Gamma(c, g)$ inserts $c$ holes and $g$ handles onto the first outgoing
        interval, completing the topological type of $f$.
\end{enumerate}

\subsection{Cycle ordering}
\label{sec:ordering}

The normal form as stated in Definition~\ref{def:nf} depends on the ordering of the
cycles $\lambda_1, \ldots, \lambda_p$.  Different orderings produce different
expressions for $\bar\sigma(f)$ and $\bar\rho(f)$, and a different arrangement of the
minimal surfaces within $K_p$.  The resulting tf-cobordisms are all related by
permutation relations in $\TFS$, so they represent the same morphism.

For the normal form to be a well-defined canonical representative (rather than a family
of equivalent expressions), we fix the following convention.

\begin{notation}\label{Notation1}
Given the interval cycles $\lambda_1, \ldots, \lambda_p$ of a connected viewable
tf-cobordism $f \colon n \to m$, let $\min(\lambda_k)$ denote the smallest label
appearing in $\lambda_k$ (using the ordering $1 < 2 < \cdots < n < 1' < 2' < \cdots
< m'$).  We order the cycles so that $\min(\lambda_1) < \min(\lambda_2) < \cdots <
\min(\lambda_p)$.  Within each cycle, the starting point is the interval with the
smallest label.
\end{notation}

With this convention, the permutations $\bar\sigma(f)$ and $\bar\rho(f)$ are uniquely
determined by $f$, and the normal form~\eqref{eq:nf} is a canonical expression in the
generators of $\TFS$.

\subsection{Minimal surface algorithm}
\label{sec:algorithm}

In this section we provide an algorithm that allows us to express any cycle of incoming
and outgoing boundaries in terms of generators of the $\TFS$ category.  Our goal is to
give a constructive procedure that converts an interval cycle into a composition of the
generators of $\TFS$; this algorithm is a key ingredient in establishing the normal form.

\begin{definition}
\label{def:adjacent}
Two intervals, $i$ and $j$, are \textbf{adjacent} if they are next to each other in an
interval cycle, that is, $(\ldots\, i\; j\, \ldots)$.
\end{definition}

\begin{algorithm}
\label{alg:cycle-to-gen}
Given a minimal TF-cobordism from $n$ intervals to $m$ intervals defined by an interval
cycle $(1 \ldots 1' \ldots)$, we express the surface in terms of the generators using
the following procedure.  We illustrate the various phases with the cycle
$(1\ 2\ 4\ 1'\ 3'\ 3\ 2')$.
\begin{enumerate}[(1)]
\item First, number each in-boundary on the left side from top to bottom in ascending
      order $(1, \ldots, n)$.
\item Next, number each out-boundary on the right side from bottom to top in ascending
      order $(1', \ldots, m')$.
\item Then, using the twists, order the left and right intervals by the order they
      appear in the interval cycle.
\end{enumerate}
\begin{center}
  \includegraphics[width=0.55\textwidth]{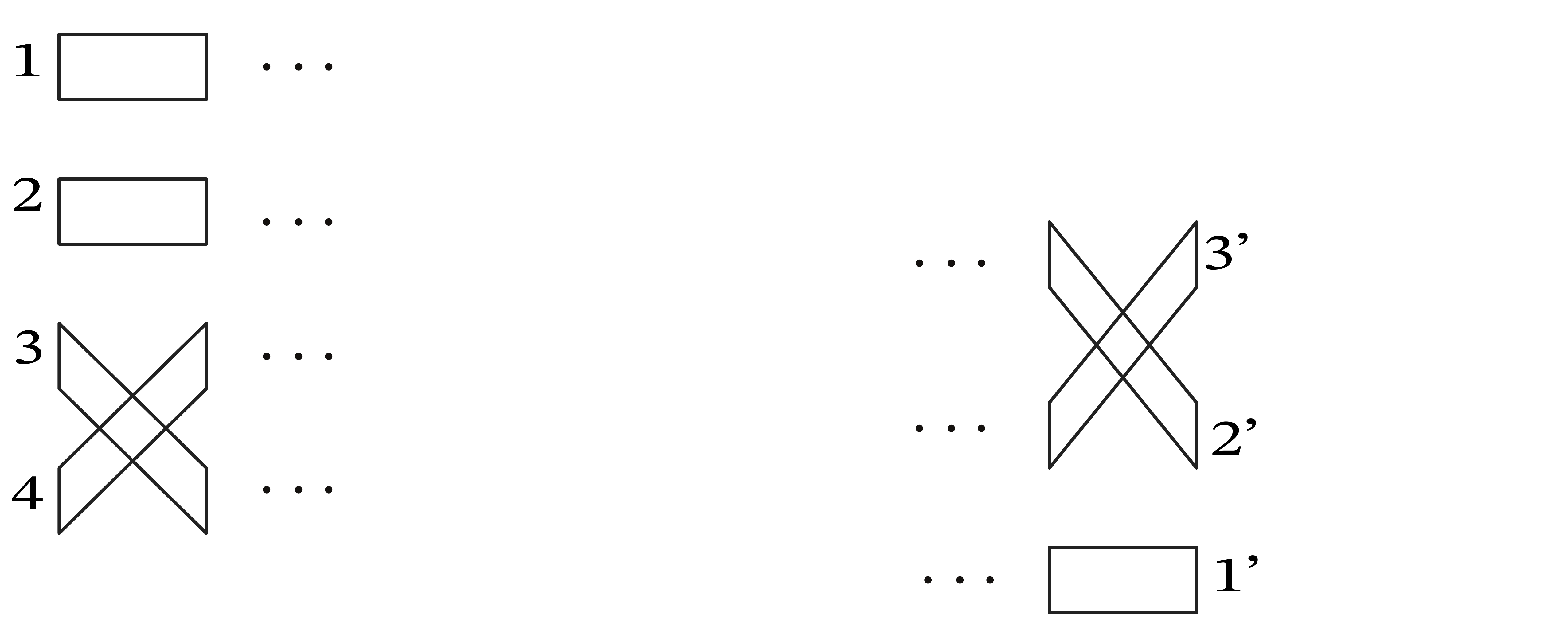}
  \captionof{figure}{Steps 1 to 3: organizing the boundaries}
\end{center}
\begin{enumerate}[(1)]
  \setcounter{enumi}{3}
\item If intervals on the same side are adjacent in the cycle, use a combination of
      pairs of pants and identities to merge these adjacent intervals to a single one.
\item Label the endpoints of each interval in the middle with arrows indicating the
      orientation and order of the cycle as in Figure~\ref{Fig:steps4-5}.
\end{enumerate}
\begin{center}
  \includegraphics[width=0.55\textwidth]{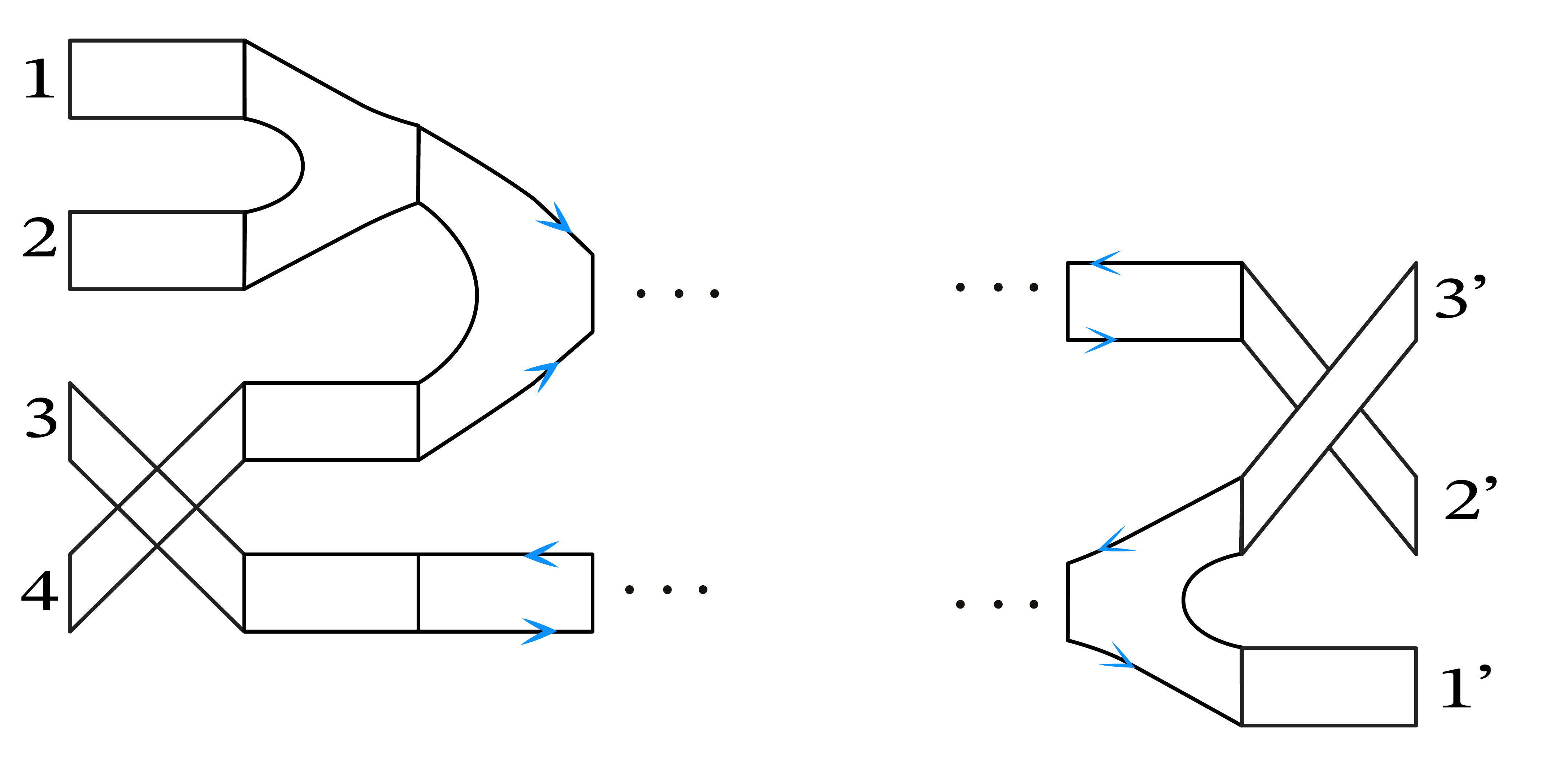}
  \captionof{figure}{Steps 4 and 5: joining consecutive boundaries from the same
  block.\label{Fig:steps4-5}}
\end{center}
\begin{enumerate}[(1)]
  \setcounter{enumi}{5}
\item Starting with the left intervals, find the arrow outgoing to $1'$ and connect
      the boundary component to the incoming arrow to $1'$.  Continue connecting these
      outgoing boundaries in numerical order until none are left.
\item Moving to the right intervals, if there are multiple outgoing boundary components,
      follow the same process as on the left until only one outgoing boundary component
      is left.  See Figure~\ref{Fig:steps6-7}.
\end{enumerate}
\begin{center}
  \includegraphics[width=0.55\textwidth]{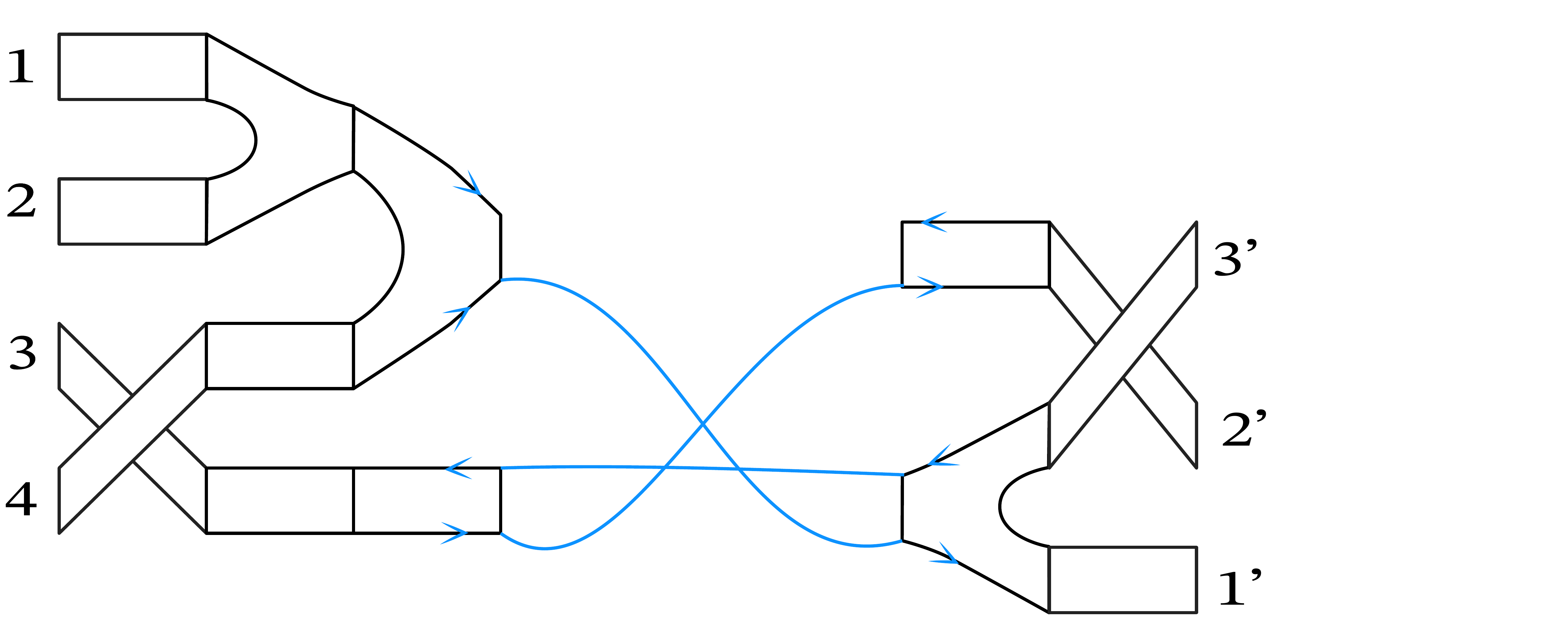}
  \captionof{figure}{Steps 6 and 7: joining boundaries except for the
  last.\label{Fig:steps6-7}}
\end{center}
\begin{enumerate}[(1)]
  \setcounter{enumi}{7}
\item For the last outgoing boundary, extend the boundary component towards the left
      in parallel to the boundary component below and trace through all the boundaries
      in this central part of the figure until it connects to the first incoming
      boundary component on the left.
\end{enumerate}
\begin{center}
  \includegraphics[width=0.5\textwidth]{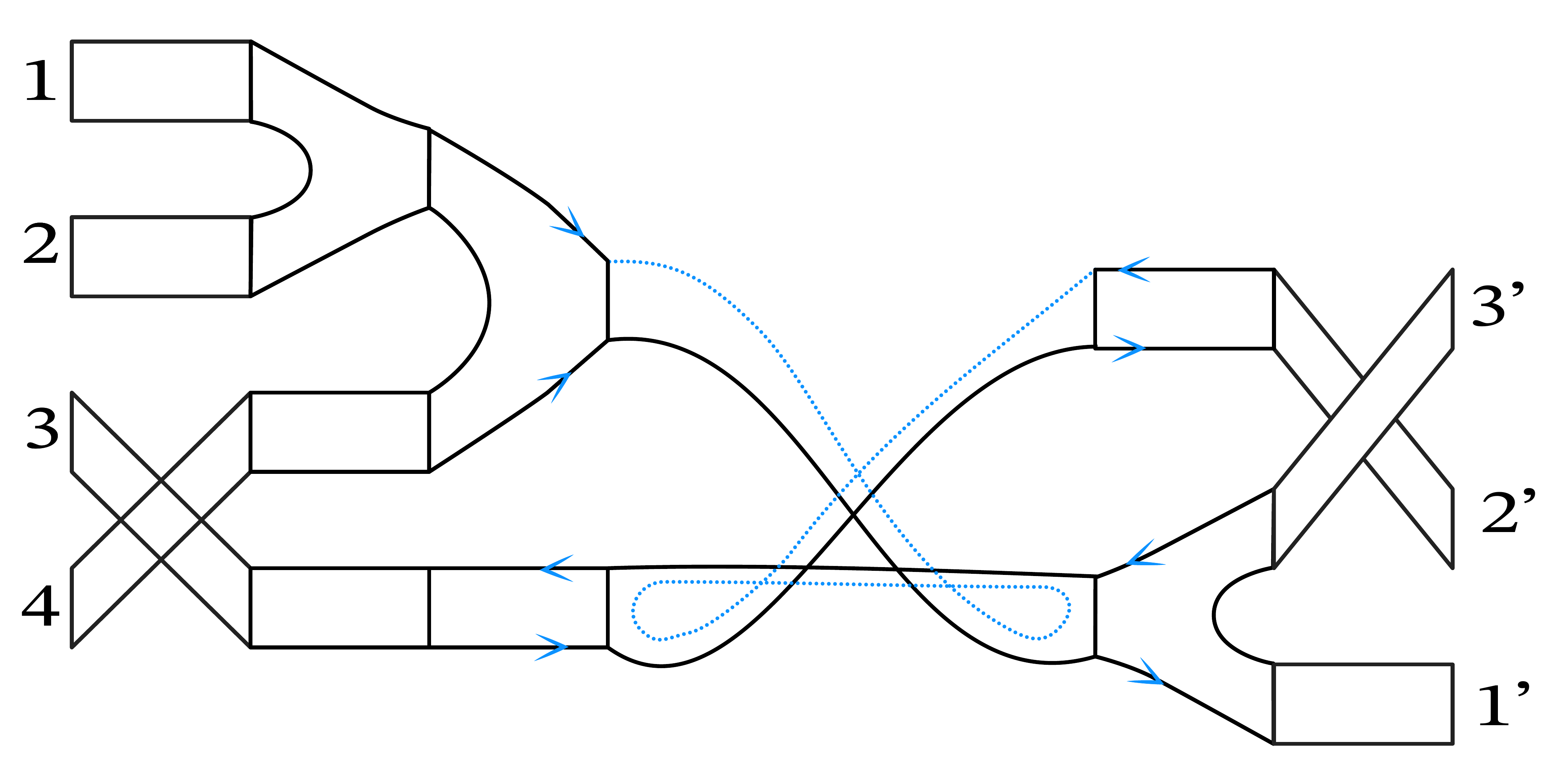}
  \captionof{figure}{Step 8: adding the last boundary component.\label{Fig:step8}}
\end{center}
Note that each time this last line traces along a vertical boundary, a multiplication
or co-multiplication and a twist are added to the figure.  Written purely in terms of
generators, the minimal surface from Figure~\ref{Fig:step8} is shown in
Figure~\ref{Fig:minimal-final}.
\begin{center}
  \includegraphics[width=0.55\textwidth]{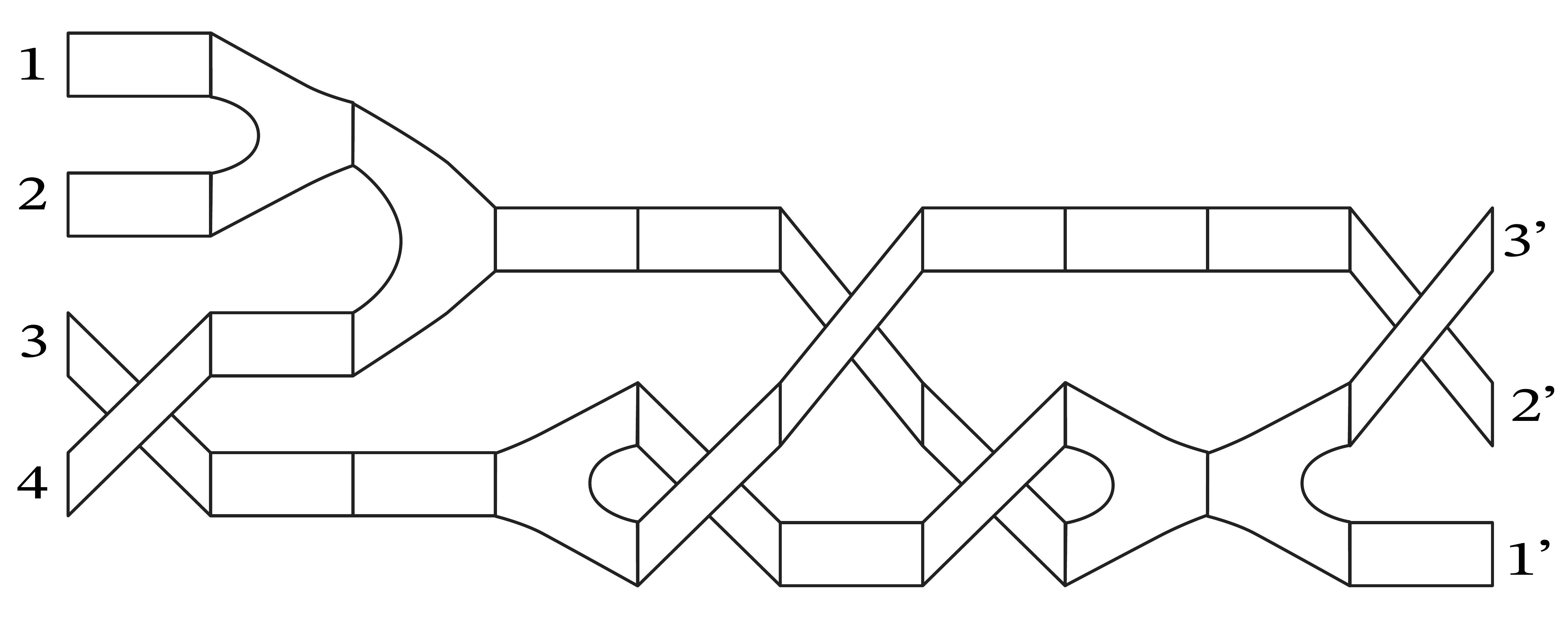}
  \captionof{figure}{Expression in terms of generators.\label{Fig:minimal-final}}
\end{center}
\end{algorithm}

\begin{remark}
The number of multiplications $\#\mu$ and co-multiplications $\#\Delta$ are invariants
of a minimal surface modulo the snake and capped pants relations.  Note that adding
additional $\mu$'s or $\Delta$'s either alters interval cycles or creates additional
boundary components.
\end{remark}

\subsection{Connecting cycles}
\label{Sec:use-of-b3}

Two thin flat surfaces expressed in terms of generators can become connected (without
modifying the structure of each cycle) by using copies of $b_3$ and additional
multiplications and co-multiplications.

There are several cases to consider according to whether the first and/or second cycles
being connected have outgoing or incoming boundaries.  See
Figure~\ref{Fig:four-cases-connect-b3}.

\begin{center}
  \includegraphics[width=0.85\textwidth]{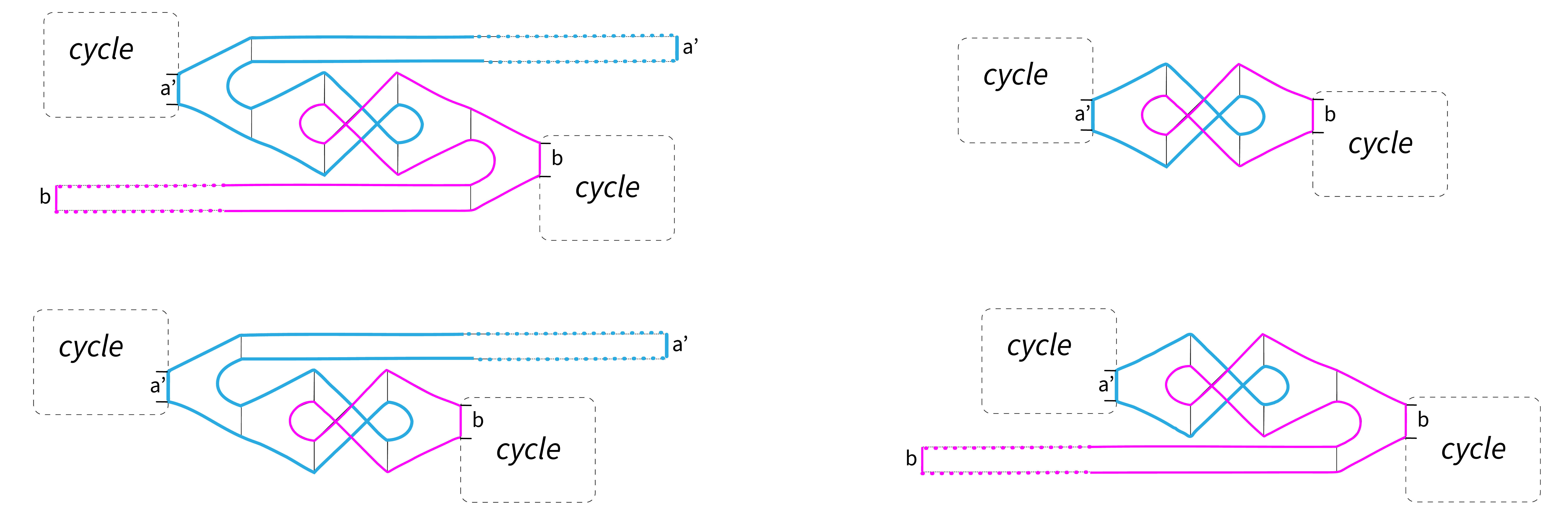}
  \captionof{figure}{Four cases of connecting cycles.  In every one of these cases the
  cycles remain unchanged, while the surfaces corresponding to each of them are now
  connected.\label{Fig:four-cases-connect-b3}}
\end{center}

Note that if a minimal surface corresponding to a cycle has no incoming boundary, it
will start with a unit; if it has no outgoing boundary, it will finish with a counit.
In either case, the unit or counit is removed from the minimal surface, and the copy of
$b_3$ is glued to the resulting free boundary.

\begin{example}
\label{ex:connecting}
The cycle $(1\ 3\ 2)$ does not have outgoing boundary.  In terms of generators it is
shown in Figure~\ref{Fig:cycle-b3}.  If we need to connect this surface to another one,
we delete the counit and glue the resulting boundary to the incoming boundary of $b_3$.
\end{example}

\begin{center}
  \includegraphics[width=0.43\textwidth]{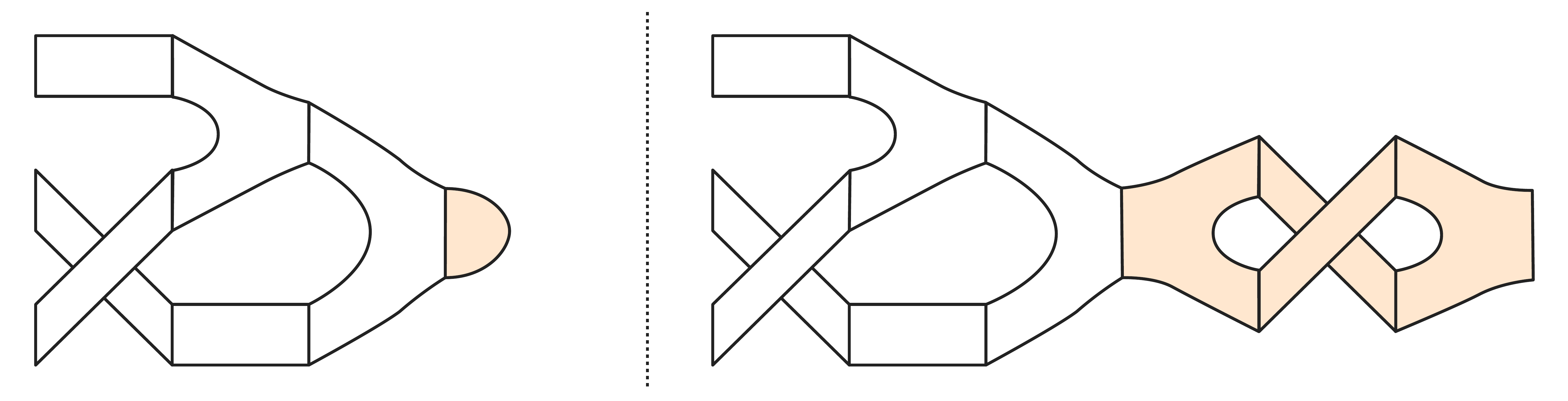}
  \captionof{figure}{Connecting a cycle with $b_3$.\label{Fig:cycle-b3}}
\end{center}

\medskip

\section{Sufficiency: form wild to normal}\label{Sec:Wild-to-normal}

This section proves that the relations of Section \ref{Sec:category-description} (numbered 1 through 10) are sufficient to reduce an arbitrary composite of generators to the normal form of \Cref{sec:normalformdef}. The argument proceeds in seven stages, each isolating one piece of the topological data of a connected tf-cobordism:

\begin{enumerate}
\item[(1)] Section \ref{subsec:excess}:count and remove every excess $\mu$ and $\Delta$ from the surface, leaving exactly the number of pants required by the topology;
\item[(2)] Section \ref{subsec:b3fact}: factor out the connector morphisms $b_3$ that join distinct interval cycles;
\item[(3)] Section \ref{sec: b1 factorization}: factor out the hole morphisms $b_1$;
\item[(4)] Section \ref{subsec:b2fact}: factor out the genus morphisms $b_2$;
\item[(5)] Section \ref{subsec:coupon}:move the factored $b_1$'s and $b_2$'s into the topology coupon at the end of the surface;
\item[(6)] Section \ref{subsec:b3order}: show that the order in which cycles are connected by $b_3$'s does not matter;
\item[(7)] Section \ref{subsec:mincycle}: arrange each remaining minimal cycle into the form produced by Algorithm \ref{alg:cycle-to-gen}.
\end{enumerate}

Together, Stages 1 through 7 reduce an arbitrary composite to the canonical form of Definition \ref{def:nf}.

\subsection{Removing Excess $\mu$'s and $\Delta$'s}\label{subsec:excess}

Consider a connected tf-cobordism $f \colon n \to m$. By traversing each boundary component of the surface, we record the interval cycles $\lambda_1, \dots, \lambda_p$, labeled following the convention of
Notation \ref{Notation1} so that $\min(\lambda_1) < \cdots < \min(\lambda_p)$, with $a_1, \dots, a_p$ incoming boundary intervals and $a'_1, \dots, a'_p$ outgoing boundary intervals, together with $c$ holes on the surface. Write $g$ for the number of $b_2$'s within the surface. Before removing the excess $\mu$'s and $\Delta$'s, we first count how many of each are required.

\subsubsection*{Counting the required pants}

Consider an arbitrary interval cycle $\lambda_i$, involving $a_i$ incoming intervals and $a'_i$ outgoing intervals. Viewed as an independent, minimal morphism $a_i \to a'_i$, the cycle can contain $\#\mu \geq a_i - 1$ and $\#\Delta \geq a'_i - 1$ (\Cref{Sec:pants-minimal-surface}). The remaining contributions to $\#\mu$ and $\#\Delta$ come from the topology of the connected surface as a whole:

\begin{itemize}
\item each non-interval boundary component (that is, each hole) can be factored out as a $b_1$. If there is a number $c$ of $b_1$'s, these contribute at least  $c$  additional $\mu$'s and $\Delta$'s;
\item each $b_2$ contributes a minimum of two $\mu$'s and two $\Delta$'s, so the $g$ many $b_2$'s contribute $2g$ additional generators of each type;
\item since the surface is connected, $(p-1)$ copies of the connector $b_3$ are required to join the interval cycles, contributing at least $(p-1)$ additional $\mu$'s and $\Delta$'s;
\item Lastly, the composite mechanism that connects one cycle to another can require more than simply a $b_3$ (as explained in section \ref{Sec:use-of-b3}). To count the additional $\Delta$'s required, we first count the number of cycles for which $a'_i = 0$, i.e., the number of cycles with no outgoing boundary intervals. Write $w$ for the number of such cycles. If disconnected from other cycles, a cycle with no outgoing intervals requires a counit, $\epsilon$, within the cycle. If such a cycle has to be connected to another one, Example \ref{ex:connecting} indicates that there is no need for additional $\Delta$ attached to the $b_3$. Conversely, for each cycle where $a'_i > 0$, a $\Delta$ is required to connect to the boundary component to the associated outgoing boundary. This results in at least $(p-1-w)$ additional $\Delta$'s required for the surface.

Similarly, we write $v$ for the number of interval cycles where $a_i = 0$.  The number of additional $\mu$'s required to connect cycles together is at least $(p-1-v)$.

\end{itemize}

Summing these contributions, the surface numbers of $\mu$'s and $\Delta$'s in any surface is bounded below by the  expressions in Equation \ref{eq:req-pants-bound}
\begin{equation}\label{eq:req-pants-bound}
\#\mu \geq \sum_{i=1}^p (a_i - 1) + c + 2g + 2(p - 1) - v, \qquad \#\Delta \geq \sum_{i=1}^p (a'_i - 1) + c + 2g + 2(p - 1) - w.
\end{equation}

This bound becomes sharp for the Normal Form of a $\TFS$ as defined in Definition \ref{def:nf}. 
Using the relations of the category described in Section \ref{Sec:category-description}, we will show in what follows that any tf-cobordism can be reduced to one which satisfies Equation \ref{eq:req pants}. 
\begin{equation}
\label{eq:req pants}
    \#\mu = \sum_{i=1}^{p} (a_i - 1) + c + 2g + 2(p - 1) - v;~ \#\Delta = \sum_{i=1}^{p} (a'_i - 1) + c + 2g + 2(p - 1) - w.
\end{equation}

\subsubsection*{The removal procedure}
\begin{algorithm}[Removing excess $\mu$'s and $\Delta$'s]\label{alg:remove-excess}
Apply the following steps to any connected tf-cobordism until it satisfies Equation \ref{eq:req pants}.
\begin{enumerate}
\item Apply relations 9 and 10, and relation 5, to any applicable composites.
\item For every remaining excess $\mu$, there must exist a composite $\Delta \circ \iota$. For each such composite, identify the $\mu$ attached to its right and apply relations 2 and 3 to pull it toward the $\Delta \circ \iota$. Exactly one of the following holds:
\begin{enumerate}[label=(\alph*)]
\item only the ``top leg" of the $\mu$ is glued to the ``bottom leg" of the $\Delta$, or vice versa: apply relation 2 and then relation 5, and the excess $\mu$ is eliminated;
\item both legs of the $\mu$ are connected to both legs of the $\Delta$, so the composite is $b_1 \circ \iota$: there must be another $\mu$ connected on the right; use relations 4 and 2 to pull this $\mu$ leftward, then relation 5 to eliminate it;
\item the $\mu$ is glued only to the ``same leg" as the $\Delta$: one of the three cases below occurs, and is resolved in Steps 3--5.
\begin{enumerate}[label=(\roman*)]
\item the composite is a saddle;
\item there is another $\Delta$ connected to both the $\mu$ and the $\Delta \circ \iota$, giving one of the two composites in \Cref{Fig:Excess}, box c-ii;
\item there is another $\Delta$ connected to both the $\mu$ and the $\Delta \circ \iota$, giving one of the two composites in \Cref{Fig:Excess}, box c-iii;.
\end{enumerate}
\end{enumerate}
\item \textbf{Resolving case (c)(i), the saddle.} If the saddle has cycle $(1~1'~2~2')$ and the $\mu$ and $\Delta$ are connected via their ``bottom legs", apply relation 11 so that the ``top legs" are glued. If instead the saddle has cycle $(1~1'~2~2')$ but the pants are connected via their ``top legs", apply relation 11 so that the ``bottom legs" are glued, placing the outer twist on the right. If the saddle has cycle $(1~2'~2~1')$, apply relation 12 to switch the legs connecting the $\mu$ and $\Delta$, placing any outer twists on the right. Finally, apply relation 10 and then relation 5 to eliminate the $\mu$ from the saddle.
\item \textbf{Resolving case (c)(ii).} This composite resembles composite (2) from the $b_3$-factorization of \Cref{subsec:b3fact}. Since we are within a single interval cycle, its two boundary components must be merged into one by a generator to the right; this can only happen if another $\mu$ connects the two outgoing intervals, producing $b_2 \circ \iota$. Identify the $\mu$ attached to the right of this composite. If it cannot be pulled past the $b_2$ using relations 3 and 2, use the $b_2$ relation to switch the saddle from $(1~2'~2~1')$ to $(1~1'~2~2')$ (or vice versa), then use relations 3 and 2 to pull the $\mu$ leftward. Apply relation 5 to eliminate it.
\item \textbf{Resolving case (c)(iii).} There must exist a $\mu$ connecting the $b_3$ and the identity, creating a $b_2$. Since the $\Delta \circ \iota$ composite corresponds to an excess $\mu$, identify this additional $\mu$ and pull it leftward using the same relations as in case (c)(ii); apply relation 5 to eliminate it.
\item Following a symmetric analog of this process, eliminate any remaining excess $\Delta$'s by manipulating all corresponding $\epsilon \circ \mu$ composites.
\end{enumerate}
\end{algorithm}

\begin{center}
  \includegraphics[width=1.05\textwidth]{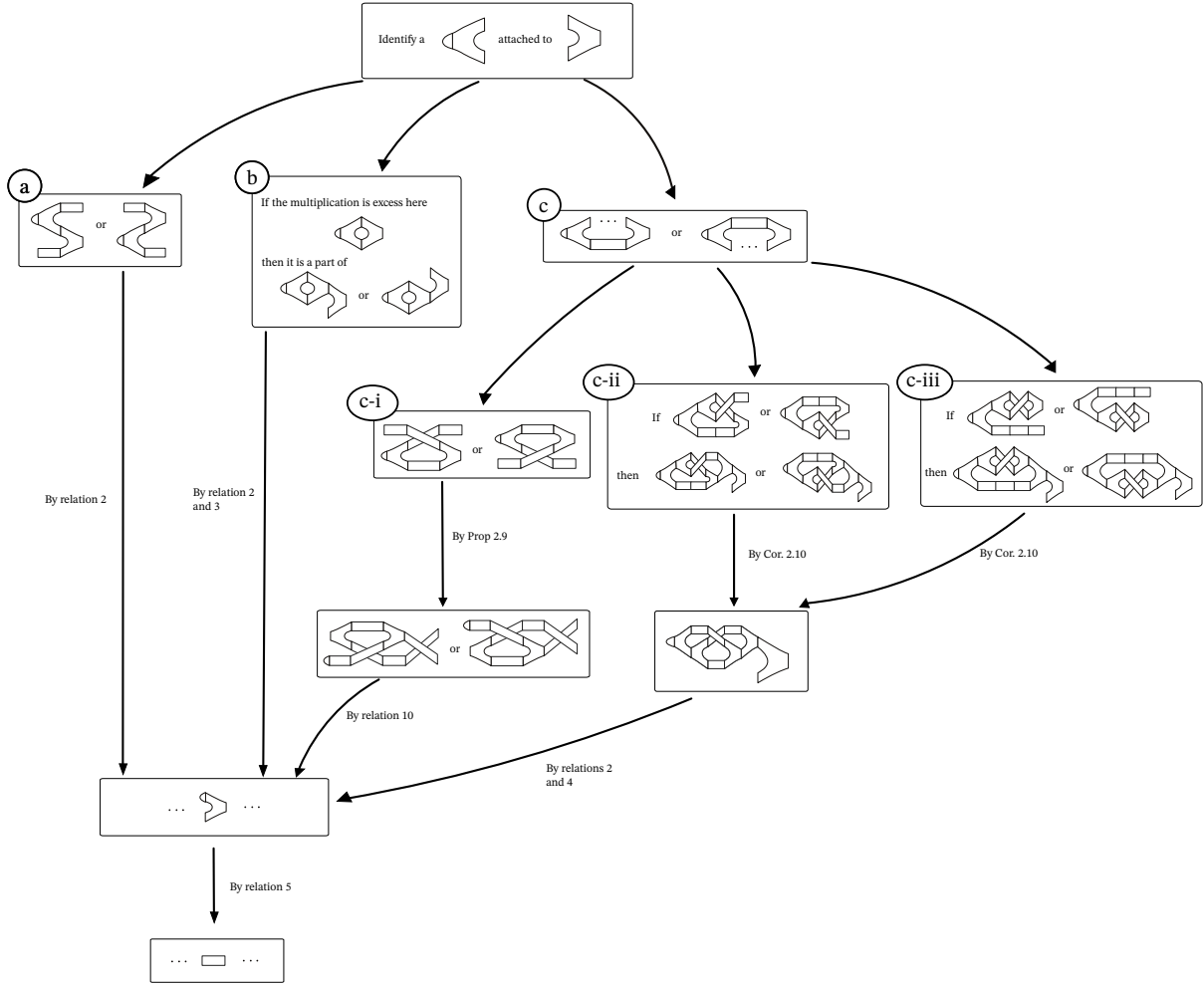}
  \captionof{figure}{Flowchart for reducing excess $\mu$'s. The flowchart for reducing excess $\Delta$'s in an analog of this one.}{\label{Fig:Excess}}
\end{center}

\subsection{Factoring out $b_3$ }\label{subsec:b3fact}

\begin{algorithm}[Factoring out $b_3$]\label{alg:b3-factorization}
Apply the following steps to a connected surface with more than one interval cycle.
\begin{enumerate}
\item Identify a connected surface component $C_1$ in which each generator touches two interval cycle boundary components; see \Cref{fig:b_3 identification} for an example. Mark the leftmost and rightmost generators of $C_1$; these are $\Delta$'s on the left and $\mu$'s on the right, and will be called the \emph{endpoints}.

  \begin{center}
           \includegraphics[width=0.3\linewidth]{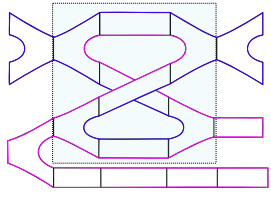}
    \captionof{figure}{The boxed composite is a $C_1$ surface component. }
    \label{fig:b_3 identification}
  \end{center}

\item Between the endpoints, check whether another surface component $C_2$ is connected to $C_1$ so that $C_2$ touches only one interval cycle boundary; see \Cref{fig:b3-c2} for an example.

\begin{center}
    \includegraphics[width=0.45\textwidth]{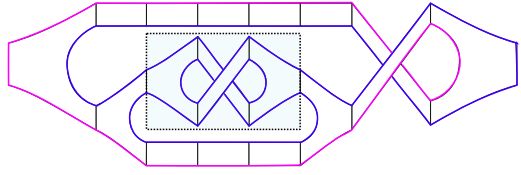}
    \captionof{figure}{The boxed composite is a  $C_2$ component.}\label{fig:b3-c2}
\end{center}

\item Let $\ell$ be the number of additional boundary components within $C_1 \cup C_2$. For each such component, factor out a $b_1$.
\item The component $C_1 \cup C_2$ is \emph{self-contained} if all of its incoming and outgoing intervals lie in $C_1$. Two cases follow.
\begin{enumerate}[label=(\alph*)]
\item \textbf{$C_1 \cup C_2$ is not self-contained.} Then there is a saddle $k$ between the two endpoints such that at least two intervals are permuted to the exterior of $C_1$; see \Cref{fig:b3-not-contained} for an example.

\begin{center}

\includegraphics[width=0.45\textwidth]{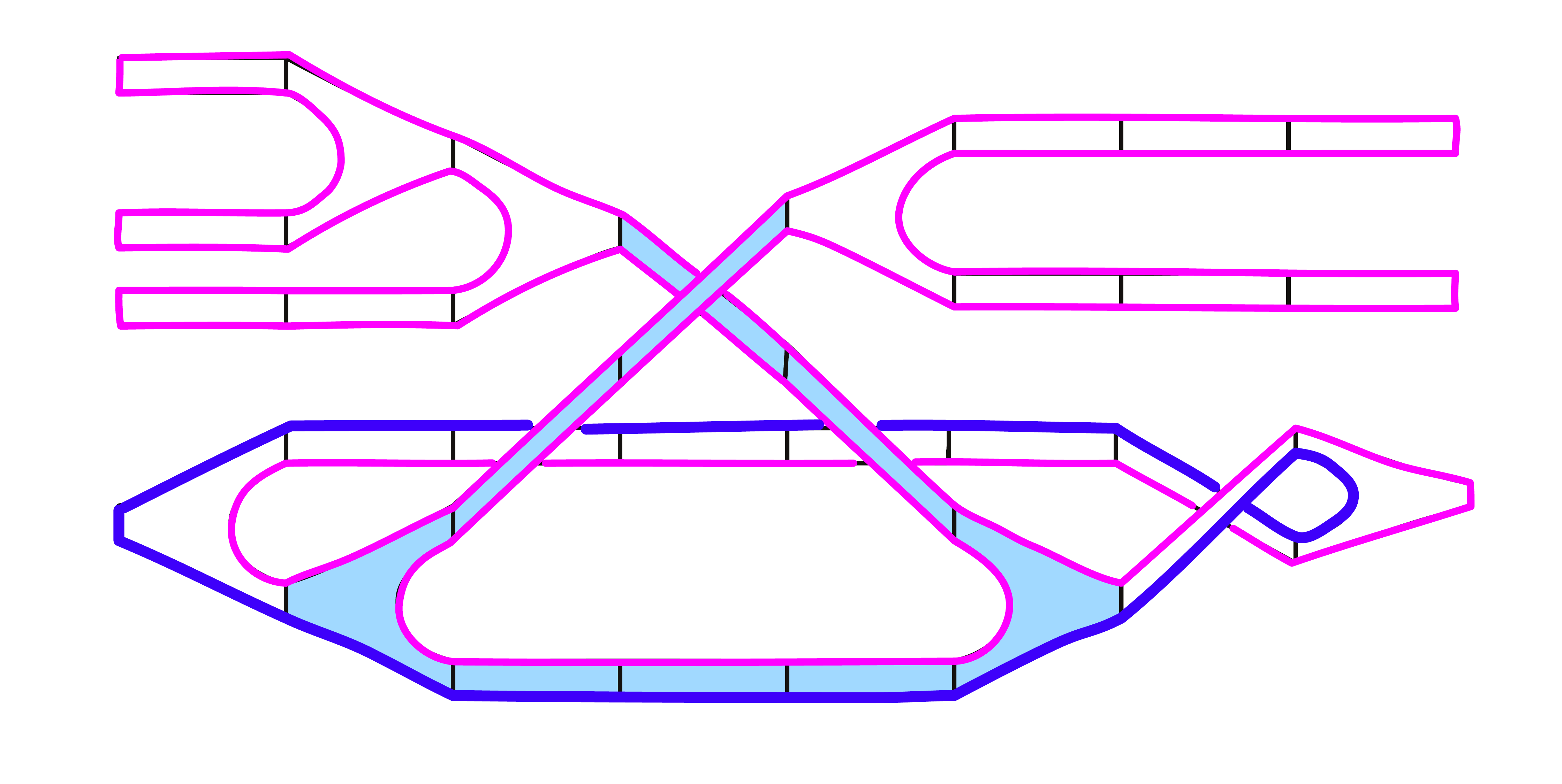}
    \captionof{figure}{The composite shaded in blue is the saddle that results in $C_1 \cup C_2$ being non-contained.}\label{fig:b3-not-contained}
\end{center}

\begin{enumerate}[label=\arabic*.]
\item Check whether any additional boundary components lie in the interior of $C_1 \cup C_2$, i.e.\ whether the hole is not permuted through the saddle. If so, apply Algorithm \ref{alg:b1-factorization} to factor out all such $b_1$'s. 
\item Look for any of the four representations of $b_2$'s from Corollary \ref{cor: b2 representations} and pull them outside of the endpoints using Lemma \ref{$b_2$ relation} together with relations 2, 3, and 4.

\item Move the endpoint(s) of $C_1$ past the saddle $k$. Using Corollary \ref{cor: saddle representations}, reduce the saddle to either $S_1$ or $S_2$. If both pants in $k$ touch both interval cycle boundaries, apply Proposition \ref{saddle relations}  so that only the $\Delta$ of $k$ touches both boundaries. Pull the left $\Delta$ endpoint of $C_1$ past $k$ using relations 2, 3, and 7, then past the $\mu$ endpoint to its right using relation 2.
\item After applying relation 2 where applicable, identify the new endpoints and repeat step (a) if the surface component is still not self-contained.
\end{enumerate}
\item \textbf{$C_1 \cup C_2$ is self-contained and $C_2 = \varnothing$.} Then the component is a morphism $j \to j$ for some $j \in \mathbb{Z}^+$. We consider different cases depending on the value of $j$:
\begin{enumerate}[label=\arabic*.]
\item If $j = 1$, the component is already a $b_3$.
\item If $j = 2$, the component is two conjoined saddles. Using relations 2, 6, and 7, pull the $\mu$'s and $\Delta$'s past each other so the two saddles sit side by side; this produces the four possible composites, up to saddle equivalence, shown in \Cref{fig:j2-possibilities}. 

   \begin{center}
 \includegraphics[width=0.5\linewidth]{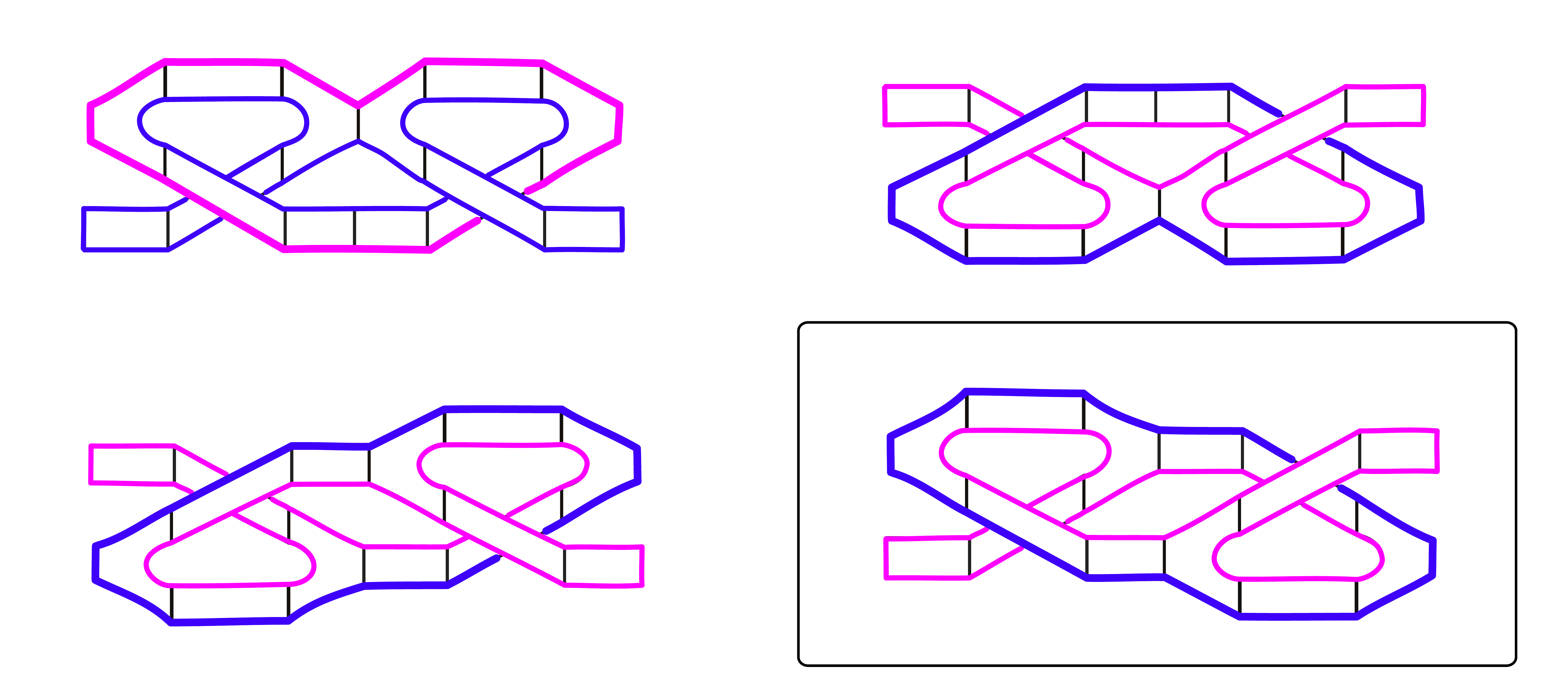}
    \captionof{figure}{Four side-by-side saddles after rearranging cojoined saddles}
    \label{fig:j2-possibilities}
\end{center}

Using relations 11 and 12 where applicable, switch to the bottom-right composite of \Cref{fig:j2-possibilities}, placing any outer twists on the left. On the right saddle, use relation 1 to swap the $\Delta$ with the identity above it; see \Cref{fig:j2-right-saddle}. 

\begin{center}
   \includegraphics[width=0.5\linewidth]{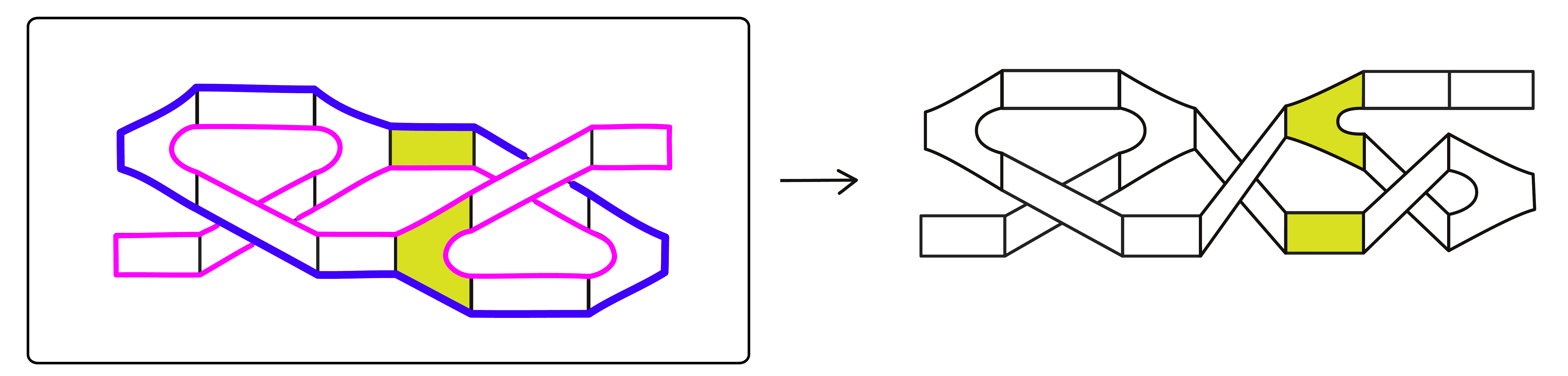}
    \captionof{figure}{Further transformations on the cojoined saddles}
    \label{fig:j2-right-saddle}
    
\end{center}

\medskip

Using relations 11 and 12, move the twist all the way to the left of the composite. Using relations 2, 6, and 7, pull the center $\mu$ to the right until it meets the rightmost $\mu$, and pull the $\Delta$ until it meets the leftmost $\Delta$, as in \Cref{fig:j2-final}.

\begin{center}
    \includegraphics[width=0.8\linewidth]{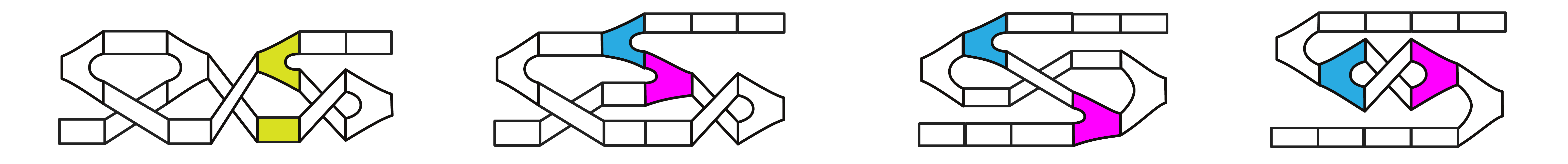}
    \captionof{figure}{Factoring out $b_3$}
    \label{fig:j2-final}
\end{center}

Finally, apply relations 3 and 4 to factor out the $b_3$.
\medskip

\item If $j > 2$, the component looks like the $j = 2$ case with $j - 2$ additional $\mu$, $\Delta$ pairs   of pants between two ``stacked" saddles; see \Cref{fig:jgt2}. 

\begin{center}
 \includegraphics[width=0.35\linewidth]{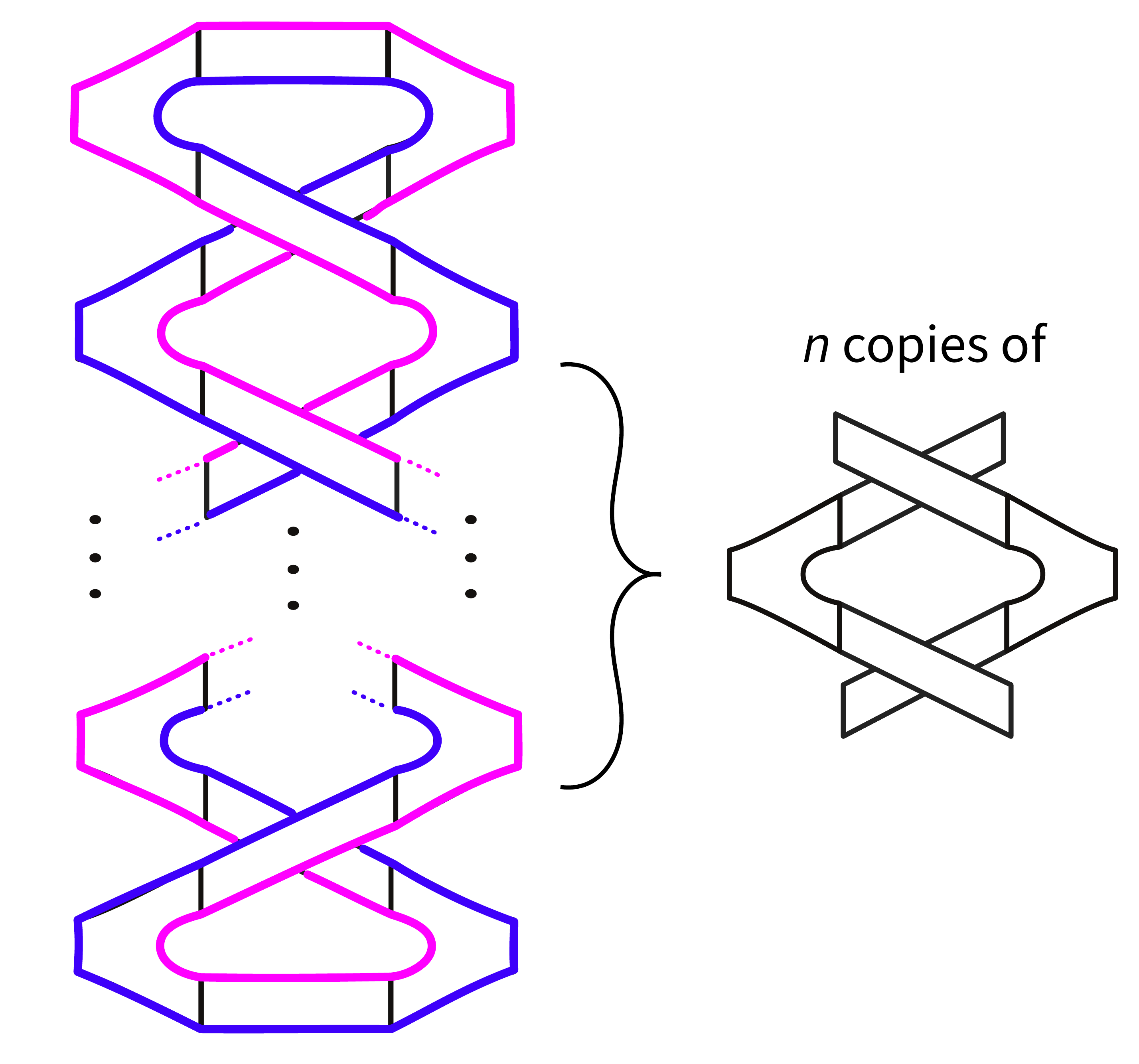}
    \captionof{figure}{Stacked saddles}
    \label{fig:jgt2}
\end{center}

Reduce to the $j = 2$ case as follows. Apply relations 6 and 7 to pull the $j - 2$ stacked $\mu$'s and $\Delta$'s to the top-left and top-right, so that relation 2 becomes applicable; see \Cref{fig:jgt2-pulling}. 

\begin{center}
   \includegraphics[width=0.35\linewidth]{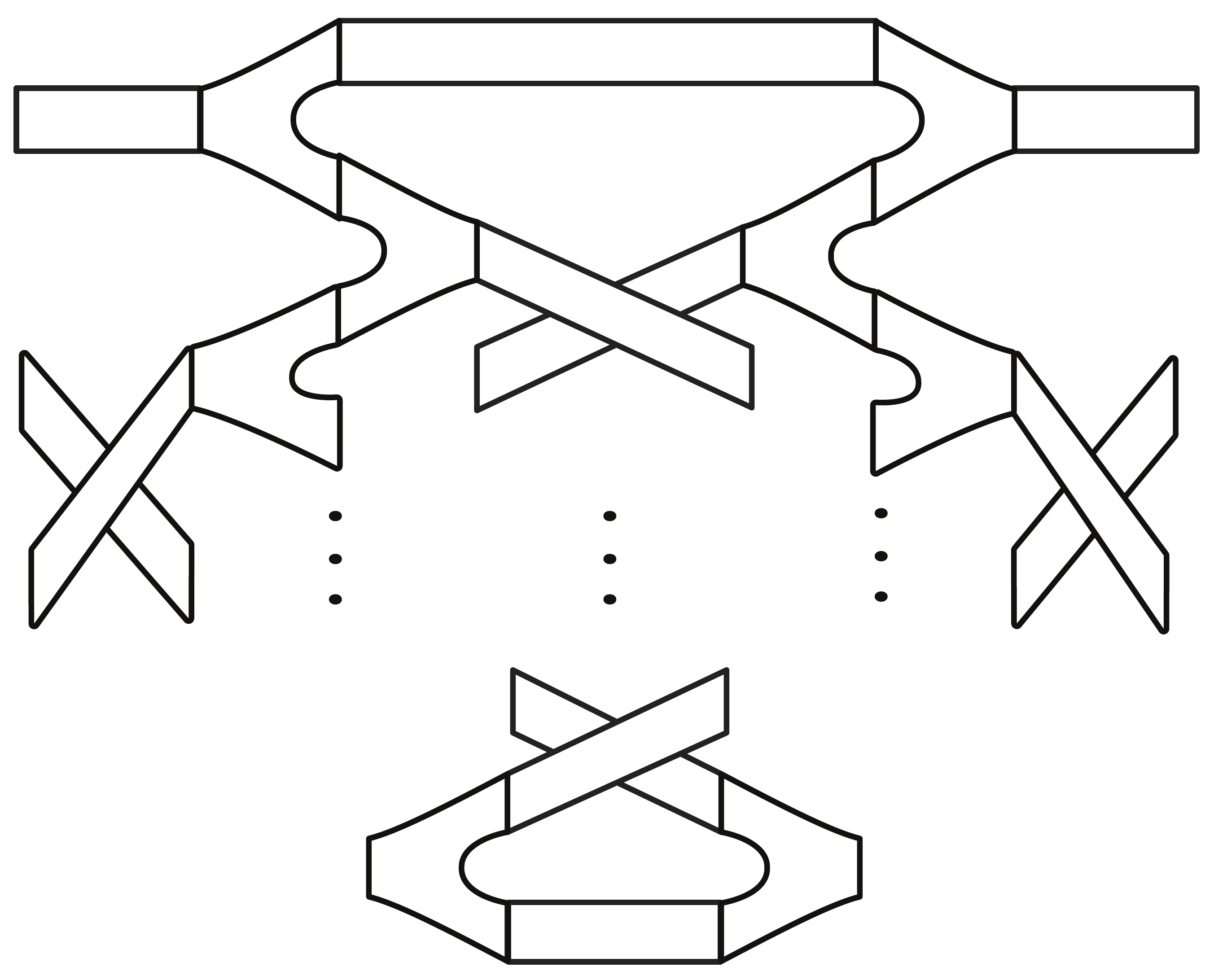}
    \captionof{figure}{Reorganizing stacked saddles}
    \label{fig:jgt2-pulling}

\end{center}

Apply relation 2 to all such composites; this will pull all $\mu$'s on the left side and all $\Delta$'s on the right side away from the saddles. Then, using relations 3, 4, 6, and 7, pull any remaining non-saddle $\mu$'s and $\Delta$'s in-between the two saddles to the left and right respectively. The resulting composite is the $j = 2$ case in the center, which we reduce accordingly as above.
\end{enumerate}
\end{enumerate}
\end{enumerate}
\end{algorithm}

The flowchart in Figure \ref{Fig: b3-factorization} gives an overview of Algorithm \ref{alg:b3-factorization}.

\begin{center}
   \includegraphics[width=0.9\linewidth]{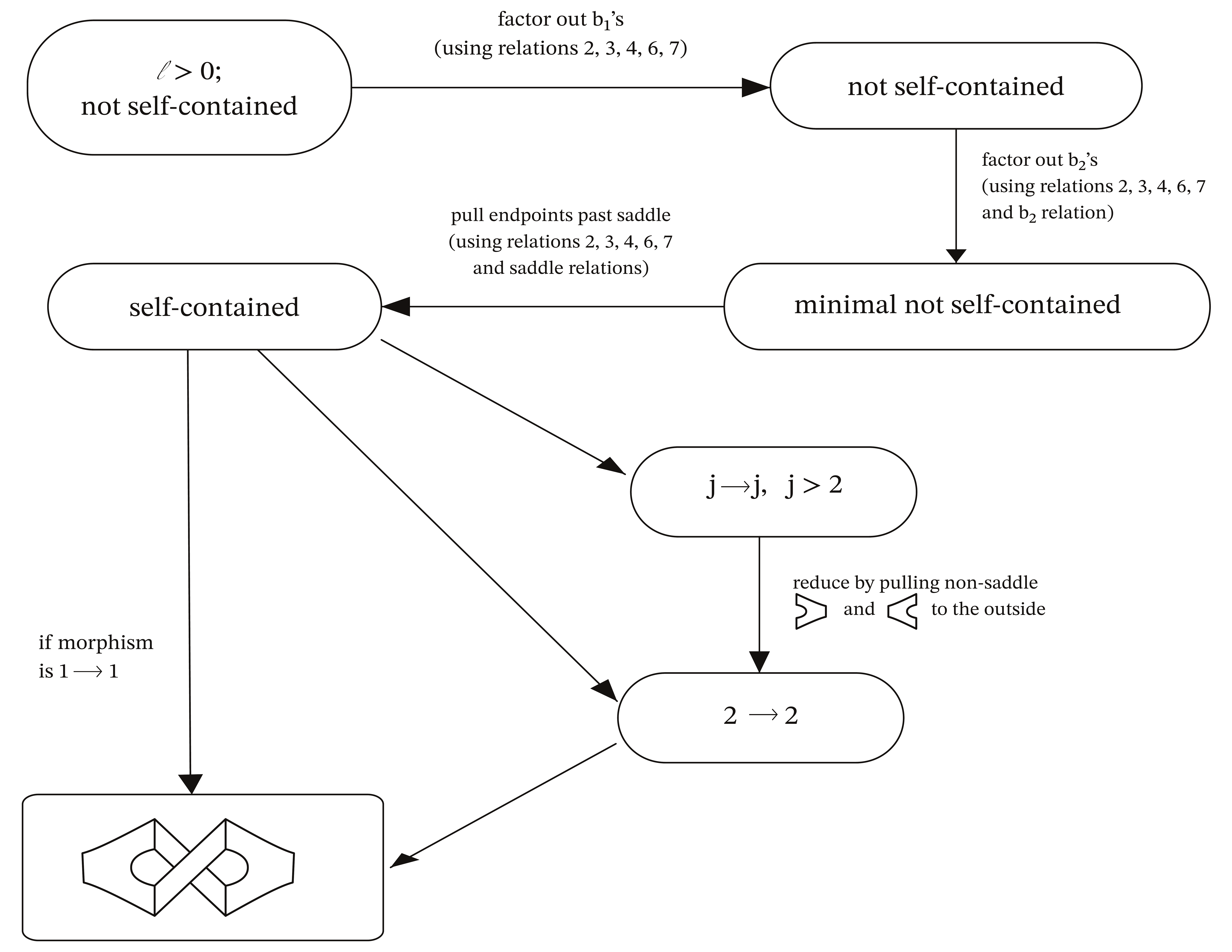}
    \captionof{figure}{Summary of procedure to factor $b_3$}
    \label{Fig: b3-factorization}

\end{center}

\subsection{Factoring out $b_1$} \label{sec: b1 factorization}

Each additional boundary component of an interval cycle corresponds to exactly one $b_1$; the following algorithm factors these out.

\begin{algorithm}[Factoring out $b_1$]\label{alg:b1-factorization}
Apply the following steps to an interval cycle with additional boundary components.
\begin{enumerate}
\item Mark all additional boundary components of the interval cycle.
\item Mark all $\mu$'s and $\Delta$'s where a marked boundary component goes on the ``inseam" of the pants;  see \Cref{fig:b_1 identification} for an example.

\
\begin{center}
     \includegraphics[width=0.55\linewidth]{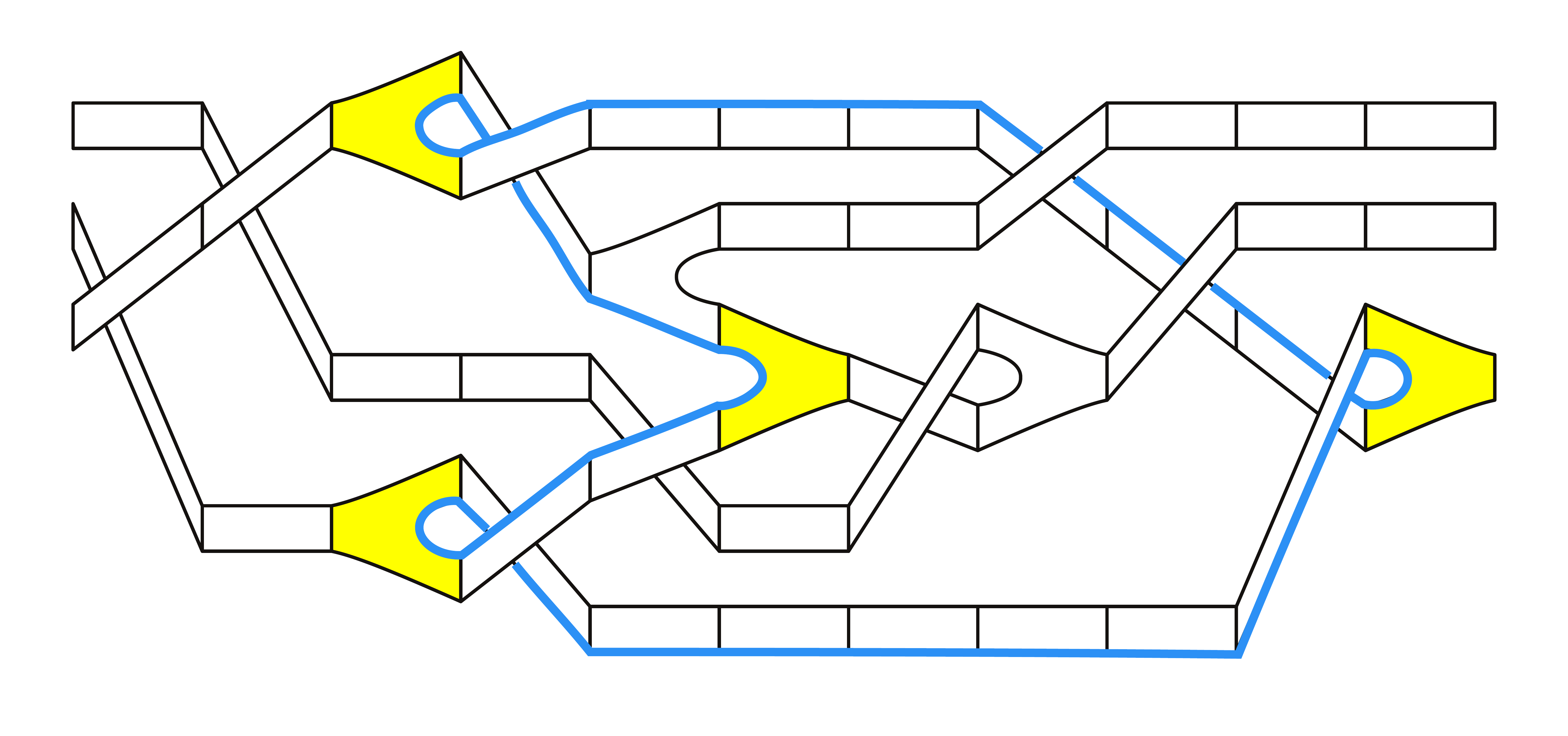}
    \captionof{figure}{The marked $\mu$'s and $\Delta$'s for this hole appear shaded in this figure}
    \label{fig:b_1 identification}
\end{center}

\item Exactly one of the following two situations occurs.
\begin{enumerate}[label=(\alph*)]
\item \textbf{Only one pair of pants is marked.} Then either a $\Delta \circ \iota$ or an $\epsilon \circ \mu$ lies inside the boundary component; see \Cref{fig:capped b3} for both possibilities. Apply relation 13 to factor out the $b_1$.

\begin{center}
        \includegraphics[width=0.4\linewidth]{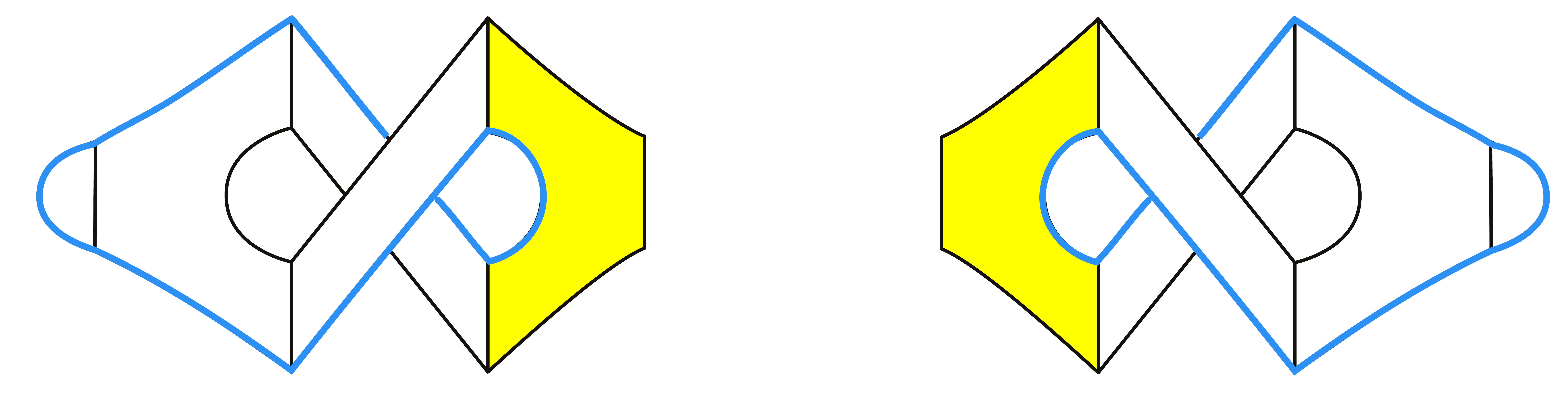}
    \captionof{figure}{Two possibilities with only one marked pant ($\Delta$ or $\mu$)}
    \label{fig:capped b3}
\end{center}

\item \textbf{Multiple pairs of pants are marked.} Then no $\Delta \circ \iota$ or $\epsilon \circ \mu$ composite is enveloped inside the boundary, so the number of marked $\mu$'s equals the number of marked $\Delta$'s. Reduce this common number to $1$ as follows:
\begin{enumerate}[label=\arabic*.]
\item Every marked $\mu$, $\Delta$ pair forms a saddle. Pick one such saddle and, using Corollary \ref{cor: saddle representations}, reduce it to either $S_1$ or $S_2$.
\item Apply Proposition \ref{saddle relations} to switch from a ``top leg"-connected saddle to a ``bottom leg"-connected saddle, placing any outer permutation morphisms away from the boundary component. This eliminates one marked $\mu$ and one marked $\Delta$.
\item Repeat steps 1--2 until exactly one marked $\mu$ and one marked $\Delta$ remain.
\end{enumerate}
Finally, pull the remaining $\mu$ and $\Delta$ together using relations 3, 4, 6, and 7,  and, if applicable, apply relation 1 to finalize the $b_1$ factorization. 
\end{enumerate}
\end{enumerate}
\end{algorithm}

\begin{center}
  \includegraphics[width=1.05\textwidth]{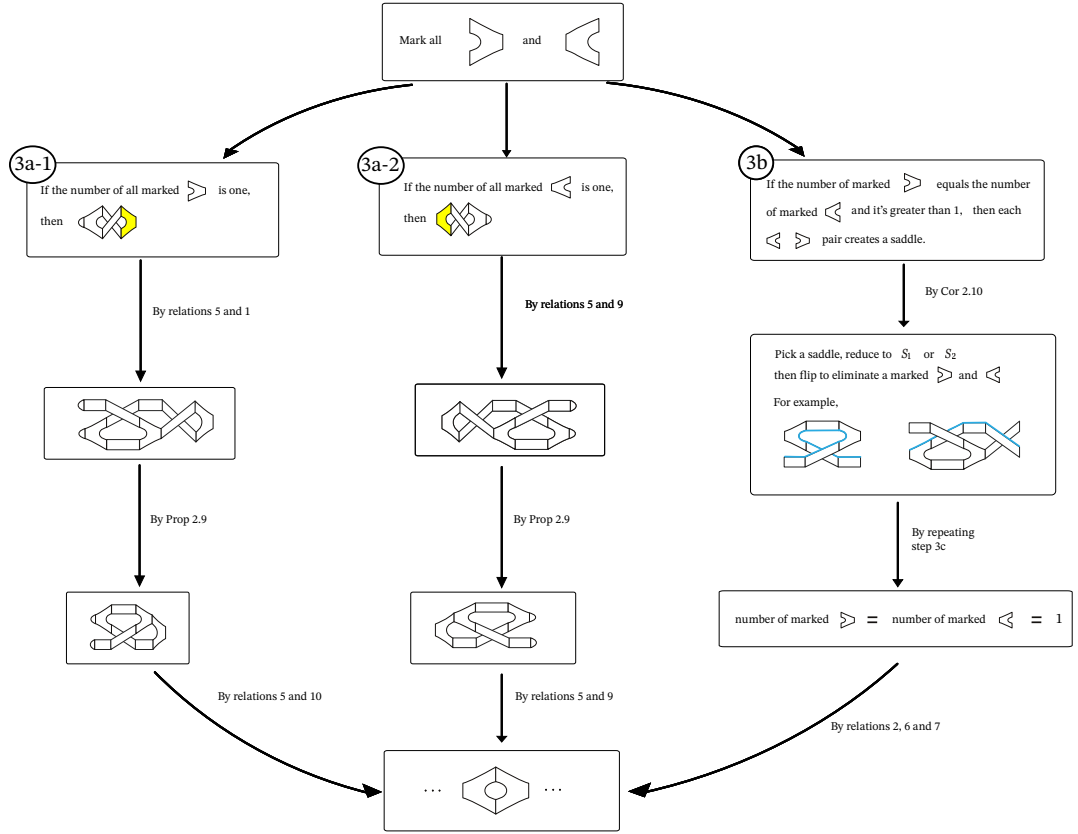}
  \captionof{figure}{Summary of the procedure to factor $b_1$s.}{\label{Fig:Excess}}
\end{center}

\subsection{Factoring out $b_2$}\label{subsec:b2fact}

    Before putting minimal cycle coupons into normal form, we must factor out all $b_2$'s from the cycle. By Corollary \ref{cor: b2 representations}, $b_2$'s have a total of four representations. Since we reduced to the minimum \#$\mu$ and \#$\Delta$ and factored out all $b_1$'s ($c$), we determine \#$b_2$ ($g$) by counting the remaining $\mu$'s and $\Delta$'s then solving for $g$ in equation \ref{eq:req pants}. 
    
\begin{algorithm}[Factoring out $b_2$]\label{alg:b2-factorization}
Suppose $g > 0$.
\begin{enumerate}
\item Check for any $b_3$'s on the interior of the cycle.
\begin{enumerate}[label=(\alph*)]
\item If any are present, each corresponds exactly to one of the two $b_3$-representations of $b_2$. Identify the $\Delta$ on the left and the $\mu$ on the right of the $b_3$ that connect via the same leg, and apply relations 2, 3, 4, and 7, where applicable, to pull these pants together and factor out the $b_2$.
\item If no $b_3$'s remain and $b_2$'s are still present, look instead for the two saddle-based representations: for each saddle within the interval cycle, check for a $\Delta$ to its left and a $\mu$ to its right that together create a $b_2$.
\end{enumerate}
\item Continue checking $b_3$'s and saddles until all $b_2$'s have been found and factored out.
\end{enumerate}
\end{algorithm}

\subsection{Topology coupon Construction}\label{subsec:coupon}

\begin{algorithm}[Assembling the topology coupon]\label{alg:coupon}
Apply the following two steps to move the factored $b_1$'s and $b_2$'s to the end of the surface.
\begin{enumerate}
\item \textbf{Placing the $b_1$'s.} To pull each $b_1$ past the minimal components of an interval cycle, use relations 2, 3, 4, and 7. If a $b_1$ meets a $b_2$ or $b_3$, Lemma \ref{lem: b1, b2, b3 commutativity} lets us move the $b_2$ (respectively $b_3$) past it. Apply these relations until all $c$ of the $b_1$'s sit at the end of interval $1'$.
\item \textbf{Placing the $b_2$'s.} To pull each $b_2$ past the minimal components, use the $b_2$ relation (Lemma \ref{$b_2$ relation}) together with relations 2, 3, 4, and 7 to move it past $\mu$'s and $\Delta$'s. By Lemma \ref{lem: b1, b2, b3 commutativity}, move all $b_2$'s past any $b_3$'s as well, until all $g$ of the $b_2$'s sit at the end of interval $1'$.
\end{enumerate}
\end{algorithm}

\subsection{Ordering of $b_3$'s }\label{subsec:b3order}

If $p>2$ i.e. there are more than two cycles in the connected surface, there are multiple ways to connect cycles via $b_3$'s. We denote the connection of two cycles, $\lambda_i,\lambda_j$, through a $b_3$ composite as $\{\lambda_i,\lambda_j\}$. For three cycles $\lambda_1, \lambda_2, \lambda_3$, there are three ways we can connect them: 
\begin{enumerate}
    \item \{\{$\lambda_1,\lambda_2$\}, \{$\lambda_1,\lambda_3$\}\};
    \item \{\{$\lambda_2,\lambda_3$\}, \{$\lambda_2,\lambda_1$\}\};
    \item \{\{$\lambda_3,\lambda_1$\}, \{$\lambda_3,\lambda_2$\}\}.
\end{enumerate}

\begin{remark}
    The way in which we can connect $p$ cycles with ($p-1$) $b_3$'s is the same as drawing a labeled tree on $p$ vertices where each vertex corresponds to a cycle and each edge corresponds to a $b_3$. The number of ways in which you can draw such a tree is given by Cayley's formula: $p^{p-2}$.
\end{remark}

\begin{lemma} \label{lem: equivalence of connecting three cycles}
    The three ways to connect three cycles $\lambda_1, \lambda_2,$  and $\lambda_3$ are equivalent by the relations.
\end{lemma}

\begin{proof}

  Figures \ref{fig: cycle connection equiv 1} and \ref{fig: cycle connection equiv 2} (given in the Appendix) prove that (modulo relation 1) any $b_3$ connection between three cycles with both incoming and outgoing boundaries are equivalent. All other cases which require different connector composites (e.g. one cycle with only incoming boundaries, another with only outgoing, and a third with incoming and going) are analogous to this case. This is because each of the connector composites (see Figure \ref{Fig:four-cases-connect-b3}) admit a representation involving a saddle; by applying Proposition \ref{saddle relations} and pulling saddles past each other, we achieve the same equivalences for all cases. If any  pairs of pants lie in-between $b_3$ composites, we can pull the pants past them by Lemma \ref{lem: b3 connector relations} and relations 2, 3, and 4.
\end{proof}

\begin{definition}\label{def: Defining R-moves}
    Define the equivalence relation $R$ on a the three possible labeled trees on a vertex set $\{i,j,k\}$: 
    $\{\{i,j\}, \{i,k\}\} \thicksim_R \{\{j,i\}, \{j,k\}\} \thicksim_R \{\{k,i\}, \{k,j\}\}$. An \emph{R-move} swaps one of the three trees with other identified by $R$.
\end{definition}

\begin{lemma}\label{lem: tree equiv by local three vertex moves}
    Let $T$ be an arbitrary labeled tree on the vertex set $V = \{1, \dots, n\}$. Let $T'$ be the labeled tree on $V$ such that $E(T') = \{\{1,2\}, \dots, \{1,n\}\}$. Then there is a finite sequence of $R$-moves that transforms $T$ into $T'$. 
\end{lemma}

\begin{proof}
Consider $T \neq T'$. then there exist two vertices, $v_1, v_2 \in V$, with a path $P = (\{1,v_1\},\{v_1,v_2\})$. Since $P$ defines a subtree on three labeled vertices, we apply $R$ such that $(\{1,v_1\},\{v_1,v_2\}) \to (\{1,v_1\},\{1,v_2\})$. Repeating this process transforms $T$ to $T'$.
\end{proof}

\begin{remark}
    Lemma \ref{lem: tree equiv by local three vertex moves} may be viewed as a restricted analogue of connectivity results for trees under edge-exchanges; see, in particular \cite[Theorem K]{Cummins}. In this context, our results follow from establishing a correspondence between $R$-moves in Definition \ref{def: Defining R-moves} and our relations.
\end{remark}

\begin{corollary}\label{cor:tree-order}
    A connected tf-cobordism with interval cycles $\lambda_1, \dots, \lambda_p$ and an arbitrary arrangement of connector composites between cycles is equivalent to the arrangement: \{\{$\lambda_1,\lambda_2$\}, $\dots$, \{$\lambda_1, \lambda_p$\}\}.
\end{corollary}

\begin{proof}
    Let $\Lambda = \{\lambda_1, \dots, \lambda_p\}$ be the collection of interval cycles of a connected tf-cobordism. Let $X = \{\{\lambda_1,\lambda_i\} \dots \{\lambda_j,\lambda_p\}\}$ be the collection of connected cycle pairs. Define a corresponding tree, $T$, where $V(T) = \Lambda$ and if $\{\lambda_c,\lambda_d\} \in X$, then $\{\lambda_c,\lambda_d\} \in E(T)$. Lemma \ref{lem: equivalence of connecting three cycles} corresponds exactly to $R$-moves defined in \ref{def: Defining R-moves} on $T$. By Lemma \ref{lem: tree equiv by local three vertex moves}, we can use Lemma \ref{lem: equivalence of connecting three cycles} on a grouping of three connected cycles to put $b_3$ connectors in normal form.
\end{proof}

\subsection{Minimal Cycle}\label{subsec:mincycle}

Lastly, we must put each minimal cycle in the form of Algorithm \ref{alg:cycle-to-gen}. Take a cycle with $a$ incoming intervals and $a'$ outgoing intervals. Let the incoming blocks of length greater than one have lengths $x_1 \dots x_t$ and let the outgoing blocks of length greater than one have lengths $x_1' \dots x_u'$.

\begin{algorithm}[Normalizing a minimal cycle]\label{alg:minimal-cycle}
Apply the following steps to place the required $\mu$'s and $\Delta$'s in each block.
\begin{enumerate}
\item If $\#\mu \neq \sum_{i=1}^t (x_i - 1)$, some block is missing a $\mu$.
\begin{enumerate}[label=(\alph*)]
\item Identify which block is missing a $\mu$, and find the $\mu$ attached to the previous block.
\item Identify the saddle point to which this $\mu$ is connected. 

\begin{enumerate}[label=--]
\item If the $\mu$ connects directly to the $\mu$ of the saddle, use relations 2, 3, and 6 to pull the $\mu$ past the saddle.
\item If the $\mu$ connects directly to the $\Delta$ of the saddle, use Proposition \ref{saddle relations} and Corollary \ref{cor: saddle representations} (according to which saddle is present) to flip the saddle so that the $\mu$'s are stacked together, placing outer twists on the left. Then apply relation 3 to pull the $\mu$ to the left until it reaches the $\Delta$ of the saddle, and apply relation 2 to pull the $\mu$ fully to the left side.
\end{enumerate}
\end{enumerate}
\item Following the symmetric process, if $\#\Delta \neq \sum_{i=1}^u (x_i' - 1)$, identify which block is missing a $\Delta$ and pull the associated $\Delta$ from the previous block in the cycle.
\end{enumerate}
\end{algorithm}

Notice that this leaves all non-block associated $\mu$'s and $\Delta$'s as saddles. Apply Proposition \ref{saddle relations} and Corollary \ref{cor: saddle representations} to reduce to $S_1$ and $S_2$.

\medskip

This completes Stage 7 of the reduction algorithm, and with it the proof of the
sufficiency theorem stated in the introduction. Combined with the normal form
theorem of Section~\ref{sec:normalform}, this gives the complete
generators-and-relations presentation of $\TFS$ from
Section~\ref{Sec:category-description}: two composites of $\iota, \varepsilon, \mu, \Delta, P$
represent the same morphism in $\TFS$ if and only if they are related by a
finite sequence of the relations of Figure~\ref{Fig:relations}, and $\TFS$ is
exhibited as the symmetric monoidal category freely generated by this
algebraic structure. 

\newpage
\appendix
\section{Details on relation moves }

\subsection{Moves in the proof of Proposition \ref{saddle relations}}

In Figure \ref{Fig:saddle-rep-sequence} we show a sequence of relations that prove the equivalence of two different representations of the saddle. The colors highlight which part of the figure is changing at each step.

    \begin{center}
      \includegraphics[width=0.75\textwidth]{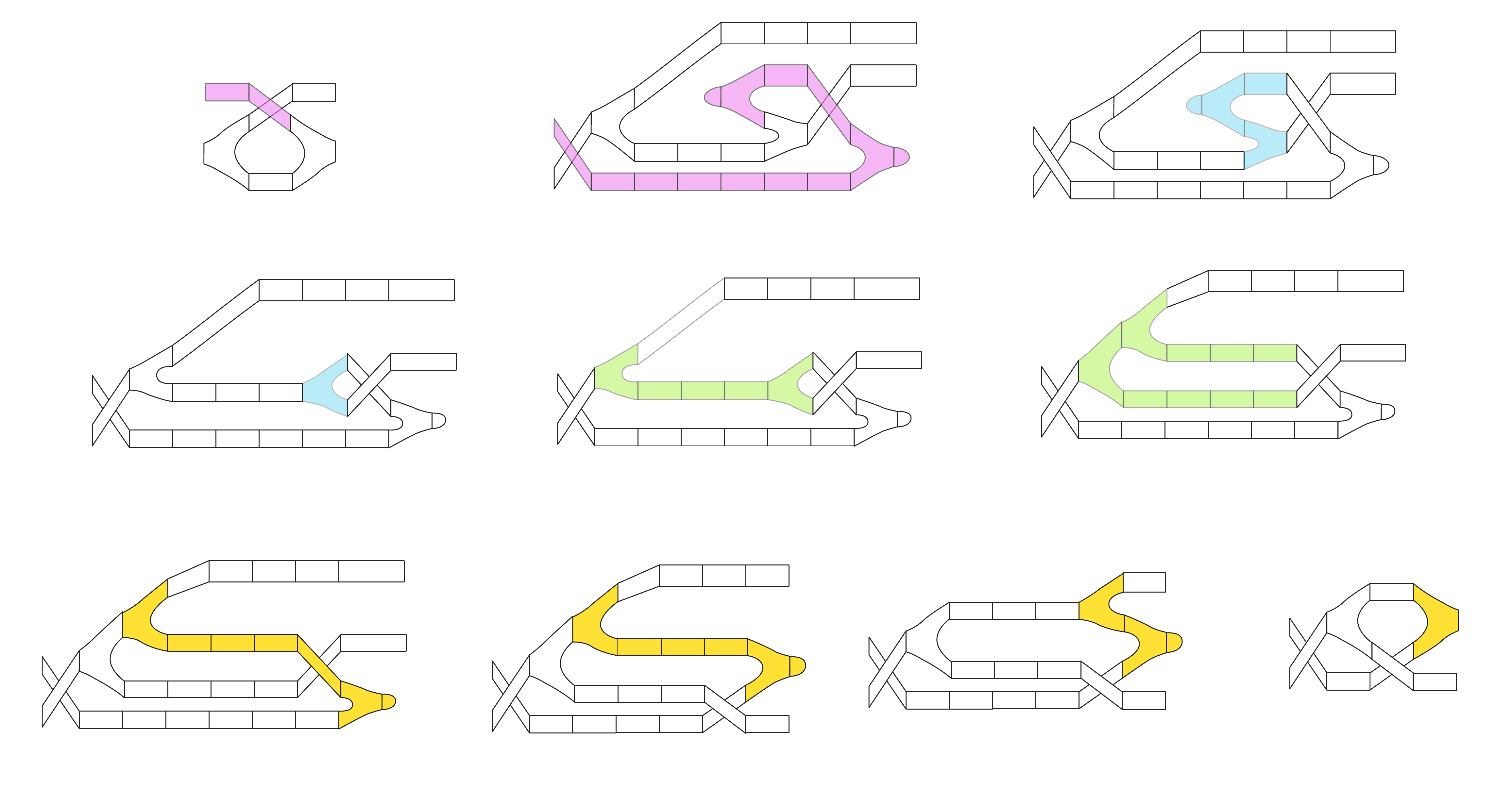}
      \captionof{figure}{Sequence of relations between two representations of the saddle\label{Fig:saddle-rep-sequence}}
    \end{center}


\subsection{Moves from the proof of Lemma \ref{$b_2$ relation}}
In Figure \ref{Fig:genus-relation} we show a sequence of relations that prove the equivalence of two different representations of the genus. 
The colors highlight which part of the figure is changing at each step.

    \begin{center}
      \includegraphics[width=0.85\textwidth]{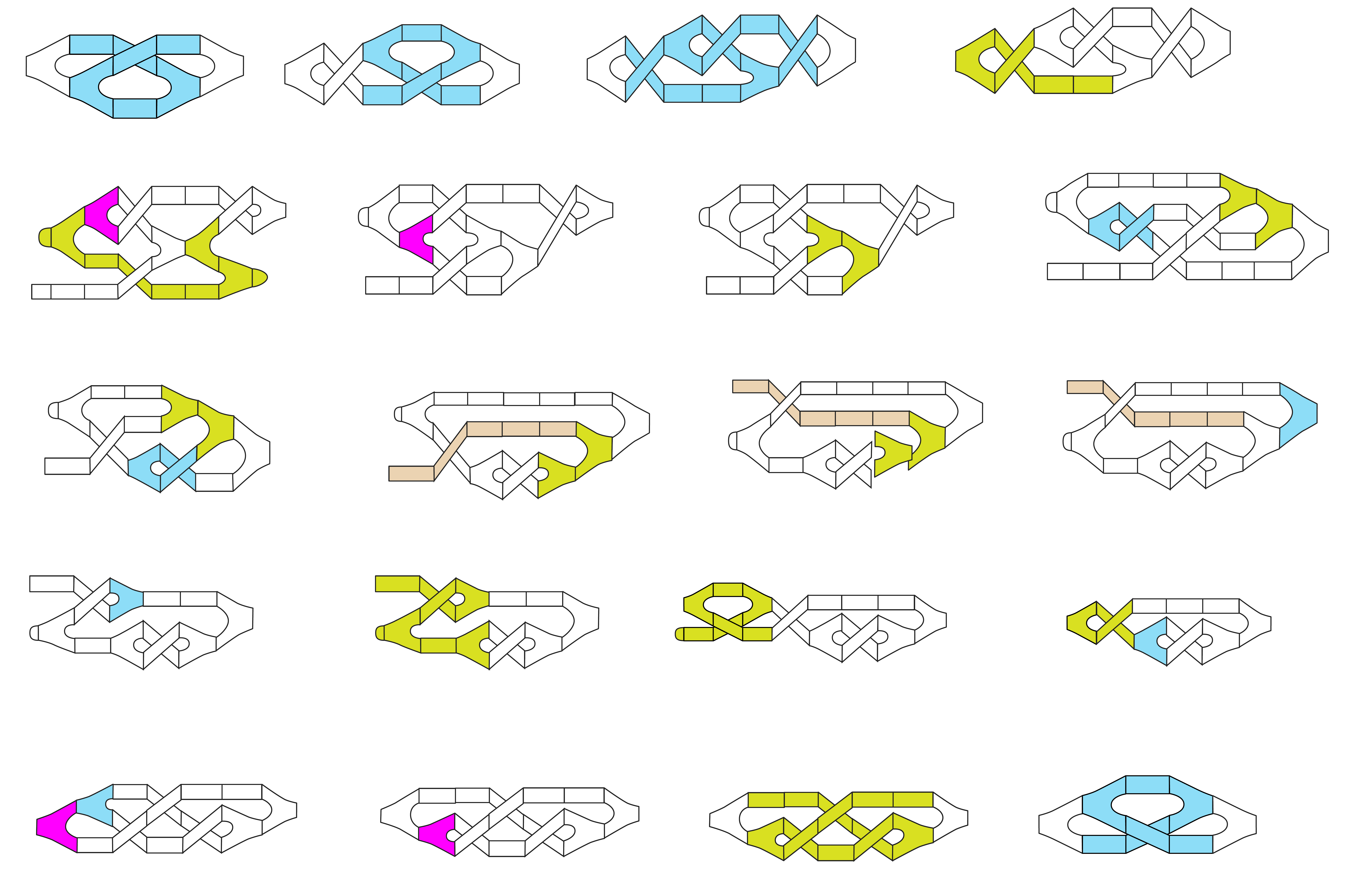}
      \captionof{figure}{Sequence of relations between two representations of the genus\label{Fig:genus-relation}}
    \end{center}

\newpage

\subsection{Figures referred to in the proof of Lemma \ref{lem: equivalence of connecting three cycles}} In Figures \ref{fig: cycle connection equiv 1} and \ref{fig: cycle connection equiv 2} we show a sequence of relations that prove the equivalence of the three different ways to connect three cycles. The colors highlight which components change at each step.

\begin{center}\label{fig: cycle connection equiv 1}
        \includegraphics[width=1.1\linewidth]{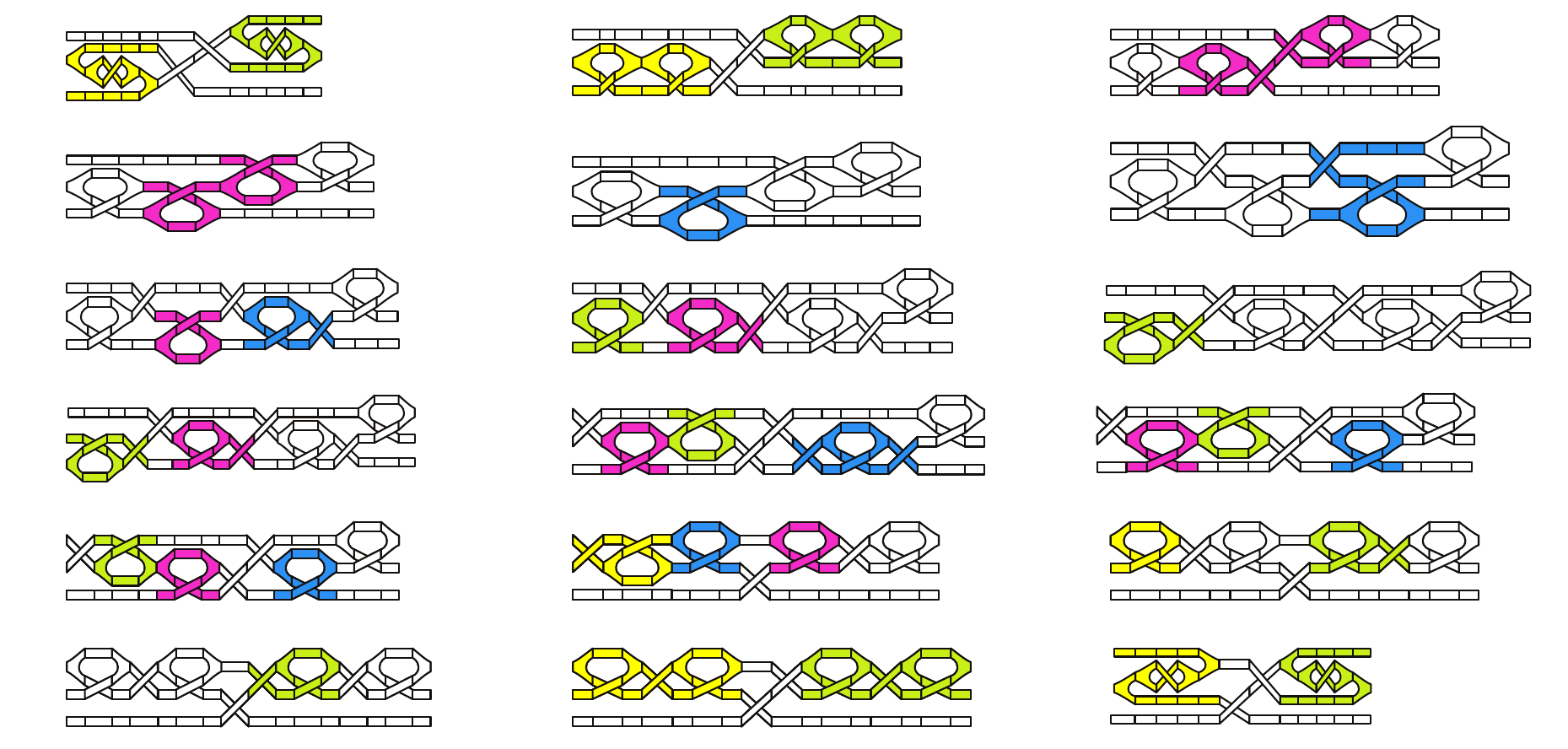}
    \captionof{figure}{}
    \label{fig: cycle connection equiv 1}
\end{center}

\begin{center}\label{fig: cycle connection equiv 2}
        \includegraphics[width=1\linewidth]{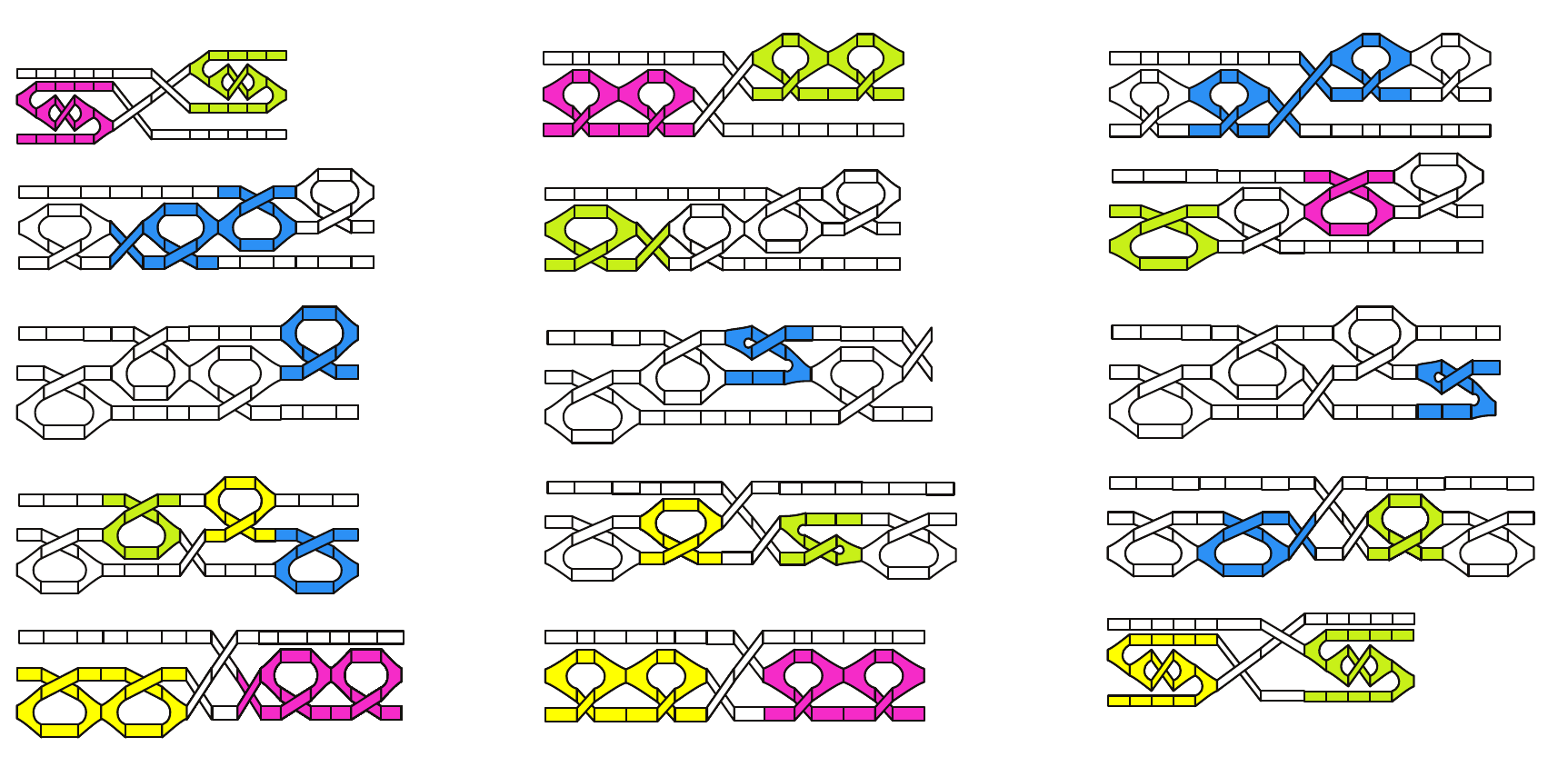}
     \captionof{figure}{}
     \label{fig: cycle connection equiv 2}
\end{center}

\newpage

\bibliographystyle{amsalpha}
{\footnotesize\bibliography{Normal/bib}}

\end{document}